\documentclass{amsart}

\usepackage{amsmath}
\usepackage{amssymb}
\usepackage{paralist}
\usepackage[english]{babel}
\usepackage[colorlinks=true]{hyperref}
\usepackage{enumitem}
\usepackage{array}
\usepackage[latin1]{inputenc}
\usepackage{xcolor}
\usepackage{esint}
\usepackage{latexsym,amsfonts}
\usepackage{amsthm}
\usepackage{cleveref}
\usepackage{mathtools}
\usepackage{verbatim}
\usepackage{graphics}

\numberwithin{equation}{section}

\newtheorem{theorem}{Theorem}[section]
\newtheorem{lemma}[theorem]{Lemma}

\newtheorem*{main-thm}{Main Theorem}
\newtheorem*{main-lemma}{Main Lemma}

\theoremstyle{definition}

\newcommand{\abs}[1]{\left\vert#1\right\vert} 
\newcommand{\pare}[1]{\left(#1\right)} 
\newcommand{\braces}[1]{\left\{#1\right\}} 
\newcommand{\brackets}[1]{\left[#1\right]} 
\newcommand{\set}[2]{\left\{#1 \; :\; #2\right\}} 
\newcommand{\prodin}[2]{\langle #1,#2 \rangle} 

\DeclareMathOperator*{\trace}{Tr}
\DeclareMathOperator*{\midr}{midrange}

\newcommand{\R}{\mathbb R}
\newcommand{\Rn}{\mathbb{R}^n}
\newcommand{\N}{\mathbb N}
\newcommand{\B}{\mathbb B}
\newcommand{\A}{\mathcal A} 
\newcommand{\e}{\mathbf{e}}
\renewcommand{\P}{\mathbf{P}}

\newcommand{\Q}{\mathbf{Q}}
\newcommand{\RR}{\mathbf{R}}
\newcommand{\I}{\mathbf{I}}
\renewcommand{\L}{\mathbf{L}}
\newcommand{\W}{\mathbf{v}\otimes\mathbf{v}}
\renewcommand{\v}{\mathbf{v}}
\newcommand{\Ln}{\mathcal{L}^{n-1}}
\newcommand{\eps}{{\varepsilon}}
\newcommand{\half}{{\frac{1}{2}}}

\def\Xint#1{\mathchoice
   {\XXint\displaystyle\textstyle{#1}}%
   {\XXint\textstyle\scriptstyle{#1}}%
   {\XXint\scriptstyle\scriptscriptstyle{#1}}%
   {\XXint\scriptscriptstyle\scriptscriptstyle{#1}}%
   \!\int}
\def\XXint#1#2#3{{\setbox0=\hbox{$#1{#2#3}{\int}$}
     \vcenter{\hbox{$#2#3$}}\kern-.5\wd0}}

\def\dashint{\Xint-}

\begin{document}


\title[Lipschitz regularity for stochastic games]{Asymptotic Lipschitz regularity for tug-of-war games with varying probabilities}

\author[Arroyo]{\'Angel Arroyo}
\address{Department of Mathematics and Statistics, University of Jyv\"askyl\"a, PO~Box~35, FI-40014 Jyv\"askyl\"a, Finland}
\email{angel.a.arroyo@jyu.fi}

\author[Luiro]{Hannes Luiro}
\email{hannes.s.luiro@jyu.fi}

\author[Parviainen]{Mikko Parviainen}
\email{mikko.j.parviainen@jyu.fi}

\author[Ruosteenoja]{Eero Ruosteenoja}
\email{eero.ruosteenoja@jyu.fi}

\date{\today}
\keywords{Dynamic programming principle, local Lipschitz estimates, stochastic games, normalized $p(x)$-Laplacian.} 
\subjclass[2010]{91A05, 91A15, 91A50, 35B65, 35J60, 35J92}

\begin{abstract}
We prove an asymptotic Lipschitz estimate for value functions of tug-of-war games with varying probabilities defined in $\Omega\subset \R^n$. The method of the proof is based on a game-theoretic idea to estimate the value of a related game defined in $\Omega\times \Omega$ via couplings.
\end{abstract}

\maketitle


\section{Introduction}

\subsection{Motivation and statement of the main result} Tug-of-war games have gained attention after the seminal papers of Peres, Schramm, Sheffield, and Wilson \cite{peresssw09,peress08}. They showed that these two-player zero-sum games have connections to homogeneous and inhomogeneous normalized PDEs in non-divergence form via dynamic programming principle (DPP for short). Regularity properties of value functions of tug-of-war games have been studied in \cite{manfredipr12,luirops13} by using translation invariance and good symmetry properties, which are no longer available in the natural generalization to the case, where probabilities depend on the location. In this space-dependent case Luiro and Parviainen \cite{luirop} showed asymptotic local H\"older regularity for value functions by developing a game-theoretic method in the spirit of couplings. Our aim is to improve this result by showing an asymptotic Lipschitz estimate.

The object of our study is the value function $u_\varepsilon:\Omega\rightarrow\R$ of the variant of tug-of-war game that is explained in Section \ref{sec:intuition} below. The function $u_\eps$ satisfies the DPP 
\begin{align}\label{DPP}
u_\eps(x)&=\frac12\sup_{\abs{\nu}=\eps}\left(\alpha(x)u_\eps(x+\nu)+\beta(x)\dashint_{B_\varepsilon^\nu(x)} \hspace{-10pt} u_\eps\ d\mathcal{L}^{n-1}\right)\nonumber\\
&\quad +\frac12\inf_{\abs{\nu}=\eps}\left(\alpha(x)u_\eps(x+\nu)+\beta(x)\dashint_{B_\varepsilon^\nu(x)} \hspace{-10pt} u_\eps\ d\mathcal{L}^{n-1}\right)
\end{align}
for $x\in\Omega$, where $\Omega\subset\R^n$ is a bounded domain, $\eps>0$, $B_\varepsilon^{\nu}(x)$ denotes the $(n-1)$-dimensional ball of radius $\varepsilon>0$ centered at $x\in\R^n$ and orthogonal to $\nu\neq 0$, and $\mathcal{L}^{n-1}$ stands for the $(n-1)$-dimensional Lebesgue measure. The coefficients $\alpha:\Omega\to(0,1]$ and $\beta:\Omega\to[0,1)$ are continuous probability functions such that $\alpha(x)+\beta(x)=1$ and
\begin{equation*}
				0<\alpha_\textrm{min}\leq\alpha(x)\leq 1
\end{equation*}
for all $x\in\Omega$.

Next suppose that, in particular, the functions $\alpha$ and $\beta$ take the form
\begin{equation*}
				\alpha(x)=\frac{p(x)-1}{n+p(x)} \quad \quad \mbox{ and } \quad \quad \beta(x)=\frac{n+1}{n+p(x)},
\end{equation*} 
where the function $p:\Omega\rightarrow(1,\infty]$ is continuous and bounded away from $1$. 

Under these assumptions, Arroyo, Heino and Parviainen \cite{ARR-HEI-PAR} showed that for a given continuous boundary data and a suitable boundary cut-off function, it holds that $u_\eps\rightarrow u$ uniformly when $\eps\rightarrow 0$, where $u$ is the viscosity solution of the normalized $p(x)$-Laplace equation $-\Delta^N_{p(x)}\, u(x)=0$. Here 
\begin{equation*}
\Delta^N_{p(x)}\, u(x):\,=\Delta u(x)+(p(x)-2)\Delta^N_{\infty}\, u(x),
\end{equation*}
where $\Delta^N_\infty$ stands for the normalized infinity Laplacian defined by
\begin{equation*}
				\Delta^N_\infty u:\,=\langle D^2 u\frac{D u}{|Du|},\frac{D u}{|Du|} \rangle. 
\end{equation*}
Moreover, by \cite[Theorem 4.1]{ARR-HEI-PAR}, the function $u_\eps$ is asymptotically H\"older continuous.

In this paper we introduce a new game-theoretic strategy to show asymptotic local Lipschitz regularity for $u_\eps$ under the assumption that the function $p(\cdot)$ is H\"older continuous. The main theorem is stated as follows.

\begin{theorem}\label{MAIN-THM}
Assume that the function $\alpha:\Omega\rightarrow(0,1]$ is H\"older continuous with a H\"older exponent $s\in (0,1)$ and a H\"older constant $C_\alpha>0$. Let $B_{2r}(x_0)\subset\Omega$ for some $r>0$. Then, for a solution $u_\varepsilon$ of \eqref{DPP} it holds
\begin{equation*}
				\abs{u_\eps(x)-u_\eps(z)}\leq C\pare{\abs{x-z}+\varepsilon} \quad\quad \mbox{ when }\ x,z\in B_r(x_0),
\end{equation*}
for some constant $C>0$ depending on $\alpha_\textrm{min}$, $C_\alpha$, $s$, $n$, $r$ and $\sup_{B_{2r}}\abs{u}$.
\end{theorem}


\subsection{Heuristic idea of the game and the method of the proof}\label{sec:intuition}

Although the proofs in this paper are mainly written without the game terminology, the intuition behind the proofs comes from the stochastic games, and this point of view helps in understanding the proofs below. The function $u_\eps$ satisfying the DPP \eqref{DPP} in $\Omega$ with some continuous boundary data is the value function of the following game. There are two players, Player I trying to maximize the payoff and Player II trying to minimize it. First the token is placed at $x_0\in \Omega$. Both players choose a vector of length $\eps$. Let $\nu^+$ be the choice of Player I and $\nu^-$ the choice of Player II. Then they flip a fair coin. If Player I wins the toss, with probability $\alpha(x_0)$ the token moves to $x_0+\nu^+$, and with probability $\beta(x_0)$, the token moves somewhere in the $(n-1)$-dimensional ball $B_\varepsilon^{\nu^+}(x_0)$ according to the uniform probability density. Similarly, if Player II wins the fair toss, with probability $\alpha(x_0)$ the token moves to $x_0+\nu^-$, and with probability $\beta(x_0)$ it moves somewhere in $B_\varepsilon^{\nu^-}(x_0)$, again according to the uniform probability density. The game continues until the token hits $\R^n\setminus \Omega$ for the first time at, let us say $x_\tau$, and then Player II pays Player I the amount given by the payoff function at $x_\tau$. Intuitively, by summing up the probabilities at $x_0$ we get the DPP \eqref{DPP} at the point $x_0$.
For a more detailed presentation of the game and its connection to the DPP \eqref{DPP}, we refer to \cite{ARR-HEI-PAR}.
 
To explain the starting point of the proof with a simple notation, we consider for a moment a more simple DPP related to the limit case $\alpha(\cdot)\equiv 1$ and $\beta(\cdot)\equiv 0$,
\begin{equation}\label{DPP-infinity}
				u_\varepsilon(x)=\frac{1}{2}\sup_{B_\varepsilon(x)}u_\varepsilon+\frac{1}{2}\inf_{B_\varepsilon(x)}u_\varepsilon,
\end{equation}
which was studied in \cite{peresssw09}, and has a connection to infinity harmonic functions. To start with, observe that $$u_{\eps}(x)-u_{\eps}(z)=\,:G(x,z)$$ can be written as a solution of a certain natural DPP in $\R^{2n}$: For all $(x,z)\in \Omega\times \Omega$ it holds that 
\begin{align}\label{eq:DPP-2}
				G(x,z)=u_\eps(x)-u_\eps(z)&=\frac{1}{2}\,\big(\,\sup_{B_\eps(x)}u_\eps+\inf_{B_\eps(x)}u_{\eps}-\sup_{B_\eps(z)}u_\eps-\inf_{B_\eps(z)}u_{\eps}\,\big)\nonumber \\
				&=\frac{1}{2}\sup_{B_{\eps}(x)\times B_{\eps}(z)}G\,+\,\frac{1}{2}\inf_{B_{\eps}(x)\times B_{\eps}(z)}G\,.
\end{align}
This resembles the  original DPP for $u_{\eps}$ in $\R^{n}$ but is for $G$ in $\R^{2n}$. In this way the question about the Lipschitz regularity of $u_{\eps}$
is converted into a question about the \text{absolute size} of a solution of \eqref{eq:DPP-2} in $\Omega\times\Omega\subset \R^{2n}$.

Next we explain the idea of estimating $|G(x,z)|$ via a stochastic game in $\R^{2n}$. We utilize the observation that $G=0$ in the diagonal set $$T:\,=\{(x,z): x=z\}.$$
The rules of the game are as follows: two game tokens are placed in $\Omega$. Two players, we and the opponent, play the game so that at each turn, if the game tokens are at $x_k$ and $z_k$ respectively, they have an equal chance to win the turn. If a player wins the turn, he can move the game token at $x_k$ to any point in $B_{\eps}(x_k)$ and the game token at $z_k$ anywhere in $B_{\eps}(z_k)$. The game stops if 1) game tokens have the same position or 2) one of the game tokens is placed outside $\Omega$. The pay-off is zero if the game ends due to the first condition  and $2\sup|u_{\eps}|$  if the game ends due to the second condition. We try to minimize the pay-off and the opponent tries to maximize the payoff. 
In other words, we try to pull the game tokens to the same position before the opponent succeeds in moving one of the game tokens outside $\Omega$. 

Heuristically speaking, the expected value of this game should evidently be
larger or equal than $|G|$ since we are using boundary values that are obviously larger than $\abs G$ at the boundary, taking the comparison principle and even existence of the value of this game for granted at this point.

Thus it suffices to estimate the value of this game. For this we need a suitable strategy in the game.  Let us consider the following natural candidate as an example: what happens if we always simply move, in the case we win the coin toss,
the game tokens straight towards each others. Indeed, if the game tokens are at $x$ and $z$, our moves are 
\begin{equation*}
				h_x:\,= -\eps\frac{x-z}{|x-z|}\quad\text{ and }\quad h_z:\,=\eps\frac{x-z}{|x-z|}\,.
\end{equation*}
It turns out that this strategy does not work well enough. The reason is that if the opponent plays against our moves but with a \text{slight turn}, by choosing 
\begin{equation*}
				\widehat{h}_x:\,= \eps T_{\theta}\bigg(\frac{x-z}{|x-z|}\bigg)\quad \text{ and }\quad \widehat{h}_z:\,= \eps T_{\theta}\bigg(-\frac{x-z}{|x-z|}\bigg),
\end{equation*} 
where $T_{\theta}$ is a rotation matrix of a very small angle $\theta$ (for $\theta\approx\eps^{3/4}$),
the distance to the boundary is expected to decrease much faster than the distance between the game tokens. Indeed, think of one step of length $\eps$ and twist $\theta$. Then in the direction $x-z$, the opponent's expected one step loss is approximately $\half \eps\theta^2=\half \eps^{5/2}$ whereas in the perpendicular direction his expected gain is $\eps \theta=\eps^{7/4}$, which is much larger for small $\eps$. 

To prevent the opponent taking advantage of the slight turn phenomenon, a more promising idea is to follow a threshold angle strategy: we could set a lower threshold and then define our strategy according to this threshold. If the step of the opponent almost taking her to $\sup_{B_{\eps}(x)\times B_{\eps}(z)}G$ makes an angle greater than the threshold with the direction $x-z$, then our strategy could be to pull the tokens straight towards each other. On the other hand, if the angle is less than the threshold, then we could pull against the step of the opponent. It turns out to be hard to evaluate the game value directly, but instead, one should try to find an explicit super-value $f$ of the game, i.e.,
\begin{equation*}
				f(x,z)>\frac{1}{2}\sup_{B_{\eps}(x)\times B_{\eps}(z)}f\,+\,\frac{1}{2}\inf_{B_{\eps}(x)\times B_{\eps}(z)}f\,, 
\end{equation*}
where $2\sup_{\partial\Omega}|u_{\eps}| \leq f$ on the boundary of $\Omega\times\Omega$, and $|f(x,z)|\lesssim|x-z|^\delta$ for some $\delta>0$. In 
\cite{luirop}, these ideas combined to a comparison argument guaranteed that the game value is less than or equal to $f$ inside $\Omega\times\Omega$ and yielded an asymptotic H\"older estimate for the function $u_\eps$ satisfying \eqref{DPP-infinity}.

To obtain an asymptotic Lipschitz estimate, on the other hand, we further need a super-value with a stronger requirement $|f(x,z)|\lesssim|x-z|$ in $\Omega\times \Omega$. This idea is applied for our proof of \Cref{MAIN-THM}. The change of the comparison function gives us substantially less advantage in choosing our strategy compared to the H\"older requirement. Hence, our threshold angle strategy cannot be fixed but it needs to depend both on the distance of the points and the H\"older exponent of the probability function $\alpha(\cdot)$. For details of our strategy, see Section \ref{Section2}.

Usually when starting from a game in $\Rn$ one could derive several different games in $\R^{2n}$, and we need to choose the  game that is suitable for our purposes. In stochastics or optimal mass transport language, we choose the couplings of the probability measures on $\Rn$ in such a way that we get a probability measure on $\R^{2n}$ having the original measures as marginals.

It has turned out that the above approach has connections to the method of couplings dating back to the 1986 paper of Lindvall and Rogers \cite{lindvallr86}, see also for example \cite{priolaw06,porrettap13,kusuoka15} for recent applications to PDEs. In the theory of viscosity solutions, this is related to the doubling of variables procedure, and in particular Ishii-Lions regularity method introduced in \cite{ishiil90}. A key point in the Ishii-Lions method is to utilize the celebrated theorem of sums at the maximum point of $u(x)-u(z)-f(x,z)$. Our proof does not rely on the theorem of sums. In addition, even if as a corollary our result also implies a similar result for the PDE,  our main objective is to prove regularity for stochastic games with nonzero step size. For example the small turn phenomenon is not present in the PDE setting. 


\subsection{Outline of the paper} In Section \ref{background} we fix the notation, introduce our super-value $f$  and state the key \Cref{PROPOSITION} for this comparison function. In Section \ref{Section2} we prove Theorem \ref{MAIN-THM} in the case $|x-z|>>\eps$, and in Section \ref{remarks} in the case $|x-z|\lesssim \eps$. Finally, in Section \ref{2<p<infty} we consider a less technical alternative game in order to prove Theorem \ref{MAIN-THM} in the restricted case $2<p(x)\leq\infty$.

\section{Preliminaries}\label{background}

\subsection{Notation}
Given $\nu\neq 0$, let
\begin{equation*}
				B_\varepsilon^\nu(x):\,=B_\varepsilon(x)\cap\braces{\nu}^\bot=\set{\xi\in\R^n}{\abs{\xi-x}<\varepsilon\quad\mbox{and}\quad\prodin{\nu}{\xi-x}=0}.
\end{equation*}
For $i=1,2,\ldots,n$, we denote by $\e_i\in\R^n$ the column vector containing $1$ in the $i$-th component and $0$ in the rest. For simplicity, we denote
\begin{equation*}
				\B^{\e_1}:\,=B_1^{\e_1}(0)=\set{\xi\in\R^n}{\abs{\xi}<1\quad\mbox{and}\quad\xi_1=0}.
\end{equation*}
Let us denote by $O(n)$ the $n$-dimensional orthogonal group
\begin{equation*}
				O(n):\,=\set{\P\in\R^{n\times n}}{\P^\top\P=\P\P^\top=\I},
\end{equation*}
where $\P^\top$ stands for the transpose of $\P$. 
Given $\nu\in\R^n$ such that $\abs{\nu}=1$, we denote by $\P_\nu\in O(n)$ an $n$-dimensional orthogonal matrix sending the vector $\e_1$ to $\nu$, that is,
\begin{equation}\label{Pnu}
				\P_\nu\e_1=\nu.
\end{equation}
Note that this is a matrix whose first column vector coincides with $\nu$ and it is not unique. Thus, due to the symmetries of the ball $\B^{\e_1}$, we can write
\begin{equation}\label{orthogo-ball}
				B_\varepsilon^\nu(x)=x+\varepsilon\,\P_\nu\B^{\e_1},
\end{equation}
with no dependence on the particular choice of the matrix $\P_\nu$ as long as \eqref{Pnu} holds. \\

\noindent \textbf{Remark.} For the rest of the paper, we fix $\eps>0$ and denote $u:\,=u_\eps$ to simplify the notation. \\

We will use the notation
\begin{equation*}
				\midr_{i\in I}a_i=\frac{1}{2}\sup_{i\in I}a_i + \frac{1}{2}\inf_{i\in I}a_i
\end{equation*}
for brevity. For the same reason we introduce the auxiliary function
\begin{equation}\label{Au}
				\A_\varepsilon u(x,\nu):\,=\alpha(x)u(x+\varepsilon\nu)+\beta(x)\dashint_{B_\varepsilon^\nu(x)} \hspace{-10pt} u(\xi)\ d\mathcal{L}^{n-1}(\xi),
\end{equation}
where $\abs{\nu}=1$. Hence, \eqref{DPP} reads as
\begin{equation}\label{DPP2}
				u(x)=\midr_{\abs{\nu}=1}\A_\varepsilon u(x,\nu),
\end{equation}
for all $x\in\Omega$. Fix any orthogonal matrix $\P_\nu$ satisfying \eqref{Pnu}, then, performing the change of variables $\zeta=\P_\nu^\top\xi$ in the integral part of \eqref{Au} and recalling \eqref{orthogo-ball}, we get
\begin{equation}\label{Au2}
				\A_\varepsilon u(x,\nu)=\alpha(x)u(x+\varepsilon\nu)+\beta(x)\dashint_{\B^{\e_1}} u(x+\varepsilon\,\P_\nu\zeta)\ d\mathcal{L}^{n-1}(\zeta).
\end{equation}

Again, we remark that the choice $\P_\nu\in O(n)$ does not play any role in \eqref{Au2}. However, the particular choice of the matrix $\P_\nu$ will become important later for obtaining estimates.


\subsection{Comparison function}\label{Comp}

For the construction of a suitable comparison function in $\R^{2n}$, first we define 
an increasing function  $\omega:[0,\infty)\to[0,\infty)$ having the desired regularity properties. 
To be more precise, let 
\begin{equation}\label{omega}
				\omega(t)=t-\omega_0\, t^\gamma \quad\quad \text{for }\quad 0\leq t\leq\omega_1:\,=\left(\frac{1}{2\gamma\omega_0}\right)^{1/(\gamma-1)}.
\end{equation}
For $t>\omega_1$, the precise formula is not relevant. 
Here $\gamma=1+s$, where $0<s<1$ is the Hölder exponent of the function $\alpha(\cdot)$, and
\begin{equation*}
				\omega_0>\frac{1}{2r^{\gamma-1}} \quad\quad \mbox{ (and thus $\omega_1<r$)} 
\end{equation*} 
is a constant depending on the function $\alpha(\cdot)$ to be fixed later (see \eqref{omega0-omega1} and \eqref{omega0}). Note that, defined in this way, $\omega$ is an increasing and strictly concave $C^2$-function in $(0,\omega_1]$. Moreover,
\begin{equation*}
				\omega'(t)\in \left[\frac12,1\right]\quad \quad \text{when}\quad 0\leq t\leq\omega_1,
\end{equation*}
and
\begin{equation}\label{omega''}
				\omega''(t)=-\gamma(\gamma-1)\omega_0\, t^{\gamma-2}<0 \quad\quad \text{for }\quad 0< t\leq\omega_1.
\end{equation}

Next, we define the function $f_1:\R^{2n}\to\R$ by
\begin{equation*}
				f_1(x,z)=C\omega(\abs{x-z})+M\abs{x+z}^2,
\end{equation*}
where $C>0$ is a constant depending on the function $C_\alpha$, $\alpha_\mathrm{min}$, $s$, $r$ and $\sup_{B_{2r}}\abs{u}$ that will be fixed later (see \eqref{C0}, \eqref{C-case1}, \eqref{C-case2}, \eqref{C-Section4-1} and \eqref{C-Section4-2}). As we have remarked, the key term in the comparison function $f_1$ is $C\omega(\abs{x-z})$, while the role of the term $M\abs{x+z}^2$ is just to guarantee that
\begin{equation}\label{ineq-f1}
				\abs{u(x)-u(z)}\leq f_1(x,z) \quad\quad \mbox{ when }\ x,z\in B_{2r}\setminus B_r,
\end{equation}
for certain $M>0$. Indeed, if $x,z\in B_{2r}\setminus B_r$ such that $\abs{x-z}\leq r$, then
\begin{equation*}
				\abs{x+z}^2=2\abs{x}^2+2\abs{z}^2-\abs{x-z}^2\geq 3r^2.
\end{equation*}
Therefore, choosing
\begin{equation}\label{M}
				M=\frac{2}{3r^2}\,\sup_{B_{2r}}\abs{u},
\end{equation}
we obtain
\begin{equation*}
				\abs{u(x)-u(z)}\leq 2\sup_{B_{2r}}\abs{u}=3Mr^2\leq M\abs{x+z}^2\leq f_1(x,z).
\end{equation*}
On the other hand, if $\abs{x-z}>r$, since $r>\omega_1$, we can extend $\omega$ outside $[0,\omega_1]$ in such a way that $\omega$ is increasing and $C\omega(r)>2\sup\abs{u}$. Then $\omega(\abs{x-z})\geq\omega(r)$ and \eqref{ineq-f1} follows.

Note that the concavity of $\omega$ turns out to be crucial when estimating the second order terms in the Taylor's expansion of $f_1$ in \Cref{Case1,Case2}. Moreover, the importance of the explicit formula for $\omega''$ \eqref{omega''} and the choice $\gamma=1+s$ is made clear in the estimate \eqref{introcite2}. To get an idea, recall that the function $\alpha(\cdot)$ is Hölder continuous with exponent $s$, that is,
\begin{equation}\label{alpha-Holder}
				\abs{\alpha(x)-\alpha(z)}\leq C_\alpha\abs{x-z}^s,
\end{equation}
for every $x,y\in\Omega$ and some $C_\alpha>0$. The coupling method leads us to estimate terms with coefficients of the type $$\frac{|\alpha(x)-\alpha(z)|}{|x-z|}$$ together with terms including $\omega''(|x-z|)$.

However, due to the discrete nature of the DPP, functions satisfying \eqref{DPP} can present jumps in the $\varepsilon$-scale. For that reason, in order to control the small scale jumps, we need to define an annular step function $f_2$ as
\begin{equation}\label{f2}
				f_2(x,z)=\begin{cases}
				C^{2(N-i)}\varepsilon & \mbox{ if }\ (x,z)\in A_i,\\
				0 & \mbox{ if }\ \abs{x-z}>\frac{N}{10}\varepsilon,
				\end{cases}
\end{equation}
where
\begin{equation*}
				A_i:\,=\set{(x,z)\in\R^{2n}}{\frac{i-1}{10}\varepsilon<\abs{x-z}\leq\frac{i}{10}\varepsilon} \quad \quad \mbox{ for }\ i=0,1,\ldots,N.
\end{equation*}
Here $N$ is a large constant depending on $C$, $\omega_0$, $C_\alpha$ and $\alpha_\mathrm{min}$ and will be chosen later (see \eqref{N1} and \eqref{N2}). Note that $f_2$ vanishes when $\abs{x-z}>\frac{N}{10}\varepsilon$ and $\sup f_2=C^{2N}\varepsilon$ is reached on the set
\begin{equation*}
				T:\,=A_0=\set{(x,z)\in\R^{2n}}{x=z}.
\end{equation*}

Therefore, our comparison function $f:\R^{2n}\rightarrow\R$ is defined as
\begin{equation*}
				f(x,z)=f_1(x,z)-f_2(x,z).
\end{equation*}
Thus, due to \eqref{f2}, the definition of $f_2$, we will use separate arguments along the proof of \Cref{MAIN-THM}, distinguishing between $f_2=0$ (\Cref{Section2}) and $f_2\neq 0$ (\Cref{remarks}).


\subsection{Statement of the key lemma for the comparison function}
Since our comparison function is $f=f_1-f_2$, where the terms in $f_1$ have been chosen such that \eqref{ineq-f1} holds and $\sup f_2=C^{2N}\varepsilon$, then
\begin{equation*}
				\abs{u(x)-u(z)}\leq f(x,z)+C^{2N}\varepsilon \quad\quad \mbox{ when }\ x,z\in B_{2r}\setminus B_r.
\end{equation*}
Then, our aim is to show that this inequality also holds in $B_r$ for properly chosen constants $C$ and $N$, that is,
\begin{equation}\label{aim}
				\abs{u(x)-u(z)}\leq f(x,z)+C^{2N}\varepsilon \quad\quad \mbox{ when }\ x,z\in B_r.
\end{equation}
This will guarantee the local Lipschitz estimate of Theorem \ref{MAIN-THM}. We will argue by contradiction. If inequality \eqref{aim} does not hold, then we can define a constant
\begin{equation}\label{contra}
				K:\,=\sup_{x',z'\in B_r}(u(x')-u(z')-f(x',z'))>C^{2N}\varepsilon. 
\end{equation}
In what follows, we may assume that $\alpha(x)\geq\alpha(z)$ because the other case follows from a symmetric argument. 
In order to obtain a contradiction, as a first step, in \Cref{LEMMA-1}, we derive lower and upper estimates for the quantity $u(x)-u(z)$ by using  the counter-assumption and the DPP \eqref{DPP2}. These estimates imply an inequality in terms of $f$, see the estimate \eqref{midrangeineq-f} below. After that, in the key \Cref{PROPOSITION}, we show precisely the opposite (strict) inequality for $f$, \eqref{f>midrange F}, mainly using the properties of the comparison function $f$, getting a contradiction and implying the desired Lipschitz estimate \eqref{aim}.

\begin{lemma}\label{LEMMA-1}
Given a function $u$ satisfying \eqref{DPP2}, suppose that the counter-assumption \eqref{contra} holds. Then, for any $\eta>0$, there exist $x,z\in B_r$ such that the comparison function satisfies
\begin{equation}\label{contra2}
				u(x)-u(z)-f(x,z)\geq K-\eta
\end{equation}
and
\begin{equation}\label{midrangeineq-f}
				f(x,z) \leq \midr_{\abs{\nu_x}=\abs{\nu_z}=1}F(f,x,z,\nu_x,\nu_z,\varepsilon) +2\eta,
\end{equation}
where $F$ is the function defined by
\begin{equation}\label{F-def}
\begin{split}
				F(x,z,\nu_x,\nu_z):\,= ~& F(f,x,z,\nu_x,\nu_z,\varepsilon):\,= \alpha(z)f(x+\varepsilon\nu_x,z+\varepsilon\nu_z) \\
				~& +\beta(x)\dashint_{\B^{\e_1}}f(x+\varepsilon\,\P_{\nu_x}\zeta,z+\varepsilon\,\P_{\nu_z}\zeta)\ d\mathcal{L}^{n-1}(\zeta) \\
				~& +(\alpha(x)-\alpha(z))\dashint_{\B^{\e_1}}f(x+\varepsilon\nu_x,z+\varepsilon\,\P_{\nu_z}\zeta)\ d\mathcal{L}^{n-1}(\zeta),
\end{split}
\end{equation}
with $\P_{\nu_x},\P_{\nu_z}\in O(n)$ satisfying $\P_{\nu_x}\e_1=\nu_x$ and $\P_{\nu_z}\e_1=\nu_z$.
\end{lemma}

\begin{proof}
By the counter-assumption \eqref{contra}, given $\eta>0$, we can immediately choose $x,z\in B_r$ so that \eqref{contra2} holds. To estimate $u(x)-u(z)$ from above, by recalling the DPP \eqref{DPP2} we have
\begin{equation}\label{u(x)-u(z)}
\begin{split}
				2[u(x)-u(z)] 
				= ~& 2\midr_{\abs{\nu_x}=1}\A_\varepsilon u(x,\nu_x) - 2\midr_{\abs{\nu_z}=1}\A_\varepsilon u(z,\nu_z) \\
				= ~& \sup_{\nu_x}\A_\varepsilon u(x,\nu_x) - \inf_{\nu_z}\A_\varepsilon u(z,\nu_z) \\
				~& + \inf_{\nu_x}\A_\varepsilon u(x,\nu_x) - \sup_{\nu_z}\A_\varepsilon u(z,\nu_z),
\end{split}
\end{equation}
where all the $\sup$ and $\inf$ are considered over the unit sphere. 
Next we look at the difference between $\A_\varepsilon u(x,\nu_x)$ and $\A_\varepsilon u(z,\nu_z)$. Using the definition \eqref{Au2}, adding and subtracting the terms
\begin{equation*}
				\alpha(z)u(x+\varepsilon\nu_x)-\beta(x)\dashint_{\B^{\e_1}}u(z+\varepsilon\,\P_{\nu_z}\zeta)\ d\Ln(\zeta),
\end{equation*}
and since $\beta(x)-\beta(z)=-(\alpha(x)-\alpha(z))$, we can write
\begin{equation*}
\begin{split}
				\A_\varepsilon u(x,\nu_x) - \A_\varepsilon u(z,\nu_z) \\
				= ~& \alpha(z)\brackets{u(x+\varepsilon\nu_x)-u(z+\varepsilon\nu_z)} \\
				~& +\beta(x)\dashint_{\B^{\e_1}}\brackets{u(x+\varepsilon\,\P_{\nu_x}\zeta)-u(z+\varepsilon\,\P_{\nu_z}\zeta)}\ d\Ln(\zeta) \\
				~& +(\alpha(x)-\alpha(z))\dashint_{\B^{\e_1}}\brackets{u(x+\varepsilon\nu_x)-u(z+\varepsilon\,\P_{\nu_z}\zeta)}\ d\Ln(\zeta),
\end{split}
\end{equation*}
for any pair of vectors $\abs{\nu_x}=\abs{\nu_z}=1$ and orthogonal matrices $\P_{\nu_x}$ and $\P_{\nu_z}$ satisfying $\P_{\nu_x}\e_1=\nu_x$ and $\P_{\nu_z}\e_1=\nu_z$. By the definition of $K$ in equation \eqref{contra}, the inequality
\begin{equation*}
				u(x')-u(z')\leq K+f(x',z')
\end{equation*}
holds for every $x',z'\in B_r$. Applying this inequality to each of the terms in the equation above we get
\begin{equation}\label{Au(x)-Au(z)}
				\A_\varepsilon u(x,\nu_x) - \A_\varepsilon u(z,\nu_z)\leq K+F(x,z,\nu_x,\nu_z),
\end{equation}
where $F$ is the function defined in \eqref{F-def}.
\\

Now, let $\abs{\widetilde\nu_x}=\abs{\widetilde\nu_z}=1$ such that
\begin{equation*}
				\begin{cases}
				\displaystyle \A_\varepsilon u(x,\widetilde\nu_x) \geq \sup_{\nu_x}\A_\varepsilon u(x,\nu_x)-\frac{\eta}{2}, \vspace{10pt} \\
				\displaystyle \A_\varepsilon u(z,\widetilde\nu_z) \leq \inf_{\nu_z}\A_\varepsilon u(z,\nu_z)+\frac{\eta}{2}.
				\end{cases}
\end{equation*}
Then, using \eqref{Au(x)-Au(z)}, we get
\begin{equation*}
\begin{split}
				\sup_{\nu_x}\A_\varepsilon u(x,\nu_x) - \inf_{\nu_z}\A_\varepsilon u(z,\nu_z)
				\leq ~& \A_\varepsilon u(x,\widetilde\nu_x) - \A_\varepsilon u(z,\widetilde\nu_z) + \eta \\
				\leq ~& K + F(x,z,\widetilde\nu_x,\widetilde\nu_z) + \eta \\
				\leq ~& K + \sup_{\nu_x,\nu_z}F(x,z,\nu_x,\nu_z) + \eta.
\end{split}
\end{equation*}

On the other hand, let $\abs{\widehat\nu_x}=\abs{\widehat\nu_z}=1$ such that
\begin{equation*}
				F(x,z,\widehat\nu_x,\widehat\nu_z)
				\leq \inf_{\nu_x,\nu_z}F(x,z,\nu_x,\nu_z) + \eta.
\end{equation*}
Hence,
\begin{equation*}
\begin{split}
				\inf_{\nu_x}\A_\varepsilon u(x,\nu_x) - \sup_{\nu_z}\A_\varepsilon u(z,\nu_z)
				\leq ~& \A_\varepsilon u(x,\widehat\nu_x) - \A_\varepsilon u(z,\widehat\nu_z) \\
				\leq ~& K + F(x,z,\widehat\nu_x,\widehat\nu_z) \\
				\leq ~& K + \inf_{\nu_x,\nu_z}F(x,z,\nu_x,\nu_z) + \eta.
\end{split}
\end{equation*}
Then, combining these estimates with \eqref{u(x)-u(z)}, we obtain
\begin{equation*}
\begin{split}
				2[u(x)-u(z)] 
				= ~& \sup_{\nu_x}\A_\varepsilon u(x,\nu_x) - \inf_{\nu_z}\A_\varepsilon u(z,\nu_z) \\
				~& + \inf_{\nu_x}\A_\varepsilon u(x,\nu_x) - \sup_{\nu_z}\A_\varepsilon u(z,\nu_z) \\
				\leq ~& 2K + \sup_{\nu_x,\nu_z}F(x,z,\nu_x,\nu_z) + \inf_{\nu_x,\nu_z}F(x,z,\nu_x,\nu_z) +2\eta.
\end{split}
\end{equation*}
Dividing by $2$ and using the midrange notation we obtain
\begin{equation*}
				u(x)-u(z) \leq K + \midr_{\abs{\nu_x}=\abs{\nu_z}=1}F(x,z,\nu_x,\nu_z) +\eta.
\end{equation*}
Finally, this together with \eqref{contra2} yields \eqref{midrangeineq-f}.
\end{proof}

Next we state the key lemma, which together with \Cref{LEMMA-1} implies the result.

\begin{lemma}\label{PROPOSITION}
Let $f$ be the comparison function and let $F$ be the function defined in \eqref{F-def}. For small enough $\eta=\eta(\varepsilon)>0$ and $x,z\in B_r$ as in \Cref{LEMMA-1}, it holds that
\begin{equation}\label{f>midrange F}
				f(x,z)>\midr_{\abs{\nu_x}=\abs{\nu_z}=1}F(f,x,z,\nu_x,\nu_z,\varepsilon)+2\eta.
\end{equation}
\end{lemma}

The proof of this lemma, which is the core of the present paper, will be presented in \Cref{Section2,remarks}, where a distinction depending on the value of $\abs{x-z}$ is made.


\section{Proof of the key \Cref{PROPOSITION}. Case $\abs{x-z}>\frac{N}{10}\varepsilon$}\label{Section2}

In this case, $f_2(x,z)=0$ and
\begin{equation}\label{f}
				f(x,z) = f_1(x,z) = C\omega(\abs{x-z})+M\abs{x+z}^2,
\end{equation}
where $x,z\in B_r$ have been fixed in \Cref{LEMMA-1} satisfying \eqref{contra2} with some fixed $\eta>0$. Next fix $|\nu_x|=|\nu_z|=1$ such that
\begin{equation*}
				F(x,z,\nu_x,\nu_z) \geq \sup_{\abs{\widehat\nu_x}=\abs{\widehat\nu_z}=1}F(x,z,\widehat\nu_x,\widehat\nu_z)-\eta.
\end{equation*}
Then, for any pair of vectors $|\widetilde\nu_x|=|\widetilde\nu_z|=1$, it holds
\begin{equation*}
				2\,\midr_{\abs{\widehat\nu_x}=\abs{\widehat\nu_z}=1}F(x,z,\widehat\nu_x,\widehat\nu_z)\leq F(x,z,\nu_x,\nu_z)+F(x,z,\widetilde\nu_x,\widetilde\nu_z)+\eta.
\end{equation*}
Thus, \Cref{PROPOSITION} will follow if we can find appropriate vectors $|\widetilde\nu_x|=|\widetilde\nu_z|=1$ such that
\begin{equation}\label{F+F-2f<0}
				F(x,z,\nu_x,\nu_z)+F(x,z,\widetilde\nu_x,\widetilde\nu_z)-2f(x,z)<-5\eta.
\end{equation}
This we will show by using Taylor's expansion.

But before this, since the explicit formula for $\omega$ given in \eqref{omega} only holds in the range $[0,\omega_1]$, first we need to choose large enough $C$ ensuring that $\abs{x-z}\leq\omega_1$. From \eqref{contra2} we have, in particular, that $u(x)-u(z)-f(x,z)>0$ and, in consequence,
\begin{equation*}
				2\sup_{B_r}\abs{u}\geq u(x)-u(z)>f(x,z)\geq C\omega(\abs{x-z}).
\end{equation*}
Since $1<\gamma<2$, we have
\begin{equation*}
				\omega(\omega_1)=\pare{1-\frac{1}{2\gamma}}\pare{\frac{1}{2\gamma\omega_0}}^{1/(\gamma-1)}
				>2\pare{\frac{1}{16\omega_0}}^{1/(\gamma-1)}.
\end{equation*}
Hence, for all
\begin{equation}\label{C0}
				C>(16\omega_0)^{1/(\gamma-1)}\sup_{B_r}\abs{u},
\end{equation}
we observe that
\begin{equation*}
				C\omega(\omega_1)>2\sup_{B_r}\abs{u}>C\omega(\abs{x-z}).
\end{equation*}
Then $\omega(\abs{x-z})\leq\omega(\omega_1)$ and $\abs{x-z}\leq\omega_1$ follows from the monotonicity of $\omega$ whenever \eqref{C0} holds. In addition, by imposing
\begin{equation}\label{omega0-omega1}
				\omega_0\geq\frac{1}{2} \quad\quad \mbox{ (and thus $\omega_1\leq1$)} ,
\end{equation}
we also ensure that $\abs{x-z}\leq 1$.


\subsection{Taylor's expansion for $F$ and game intuition}

First, we need to compute the second order Taylor's expansion of $f(x+h_x,z+h_z)$, where $h_x$ and $h_z$ denote column vectors in $\R^n$. For that purpose, we start by introducing the following notation, that will be useful in what follows: for fixed $x\neq z$ in $B_r$, let
\begin{equation*}
				\v:\,=\frac{x-z}{\abs{x-z}},
\end{equation*}
and denote by $V$ the vector space $V=\mathrm{span}\braces{\v}$. Given $h\in\R^n$, we denote by $h_V$ the projection of $h$ on the space $V$ and by $h_{V^\bot}$ the modulus of the projection on the $(n-1)$-dimensional space of vectors orthogonal to $\v$. That is, 
\begin{equation}\label{hV}
\begin{split}
				& h_V:\,=\prodin{\v}{h}, \\
				& h_V^2:\,=\prodin{\v}{h}^2
				=\trace\braces{\W\cdot h\otimes h},\\
				& h_{V^\bot}^2 :\,=\abs{h}^2-h_V^2
				=\trace\braces{(\I-\W)\cdot h\otimes h}.
\end{split}
\end{equation}

\begin{lemma}
Let $f$ be the comparison function \eqref{f}. Then the second order Taylor's expansion of $f$ is
\begin{multline}\label{TAYLOR0000}
				f(x+h_x,z+h_z)-f(x,z) \\
\begin{split}
				\leq ~& C\,\omega'(\abs{x-z})(h_x-h_z)_V+2M\prodin{x+z}{h_x+h_z} \\
				~& + \frac{C}{2}\,\omega''(\abs{x-z})(h_x-h_z)_V^2 + \frac{C}{2}\,\frac{\omega'(\abs{x-z})}{\abs{x-z}}(h_x-h_z)_{V^\bot}^2 \\
				~& + (4M+1)\abs{x-z}^{\gamma-2}\,\varepsilon^2,
\end{split}
\end{multline}
for every $\abs{h_x},\abs{h_z}\leq\varepsilon$.
\end{lemma}

\begin{proof}
We need to compute each term in the second order Taylor's expansion
\begin{multline}\label{TAYLOR00}
				f(x+h_x,z+h_z)-f(x,z) \\
				=\prodin{Df(x,z)}{
				\brackets{\begin{array}{c}
				h_x \\
				h_z
				\end{array}}}
				+\frac{1}{2}\prodin{D^2f(x,z)
				\brackets{\begin{array}{c}
				h_x \\
				h_z
				\end{array}}}{\brackets{\begin{array}{c}
				h_x \\
				h_z
				\end{array}}}
				+\mathcal{E}_{x,z}(h_x,h_z).
\end{multline}
For that reason, we will make use of the formulas for the gradient and the Hessian
\begin{equation*}
				D_x\abs{x-z}=\frac{x-z}{\abs{x-z}} \quad \quad \mbox{ and } \quad \quad D_{xx}\abs{x-z}=\frac{1}{\abs{x-z}}\pare{\I-\frac{x-z}{\abs{x-z}}\otimes\frac{x-z}{\abs{x-z}}},
\end{equation*}
for $x\neq z$. Then, since $\v=\dfrac{x-z}{\abs{x-z}}$ and differentiating \eqref{f} we get
\begin{equation*}
\begin{split}
				& D_xf(x,z)=C\,\omega'(\abs{x-z})\v+2M(x+z), \\
				& D_zf(x,z)=-C\,\omega'(\abs{x-z})\v+2M(x+z), \\
				& D_{xx}f(x,z)=D_{zz}f(x,z)=\L+2M\I \quad \mbox{and} \\
				& D_{xz}f(x,z)=-\L+2M\I,
\end{split}
\end{equation*}
where
\begin{equation*}
				\L=C\,\omega''(\abs{x-z})\W+C\,\frac{\omega'(\abs{x-z})}{\abs{x-z}}(\I-\W).
\end{equation*}
Thus we obtain
\begin{equation*}
				Df(x,z)=C\,\omega'(\abs{x-z})
				\brackets{\begin{array}{c}
				\v \\
				-\v
				\end{array}}
				+2M\brackets{\begin{array}{c}
				x+z \\
				x+z
				\end{array}}
\end{equation*}
and
\begin{equation*}
				D^2f(x,z)=
				\brackets{\begin{array}{cc}
				\L & -\L \\
				-\L & \L
				\end{array}}
				+2M\brackets{\begin{array}{cc}
				\I & \I \\
				\I & \I
				\end{array}}.
\end{equation*}
Plugging these into the terms in \eqref{TAYLOR00} yields
\begin{equation*}
				\prodin{Df(x,z)}{\brackets{\begin{array}{c}
				h_x \\
				h_z
				\end{array}}}=C\,\omega'(\abs{x-z})\prodin{\v}{h_x-h_z}+2M\prodin{x+z}{h_x+h_z}
\end{equation*}
and
\begin{equation*}
\begin{split}
				\frac{1}{2}\prodin{D^2f(x,z)\brackets{\begin{array}{c}
				h_x \\
				h_z
				\end{array}}}{\brackets{\begin{array}{c}
				h_x \\
				h_z
				\end{array}}}
				= ~& \frac{1}{2}\trace\braces{D^2f(x,z)
				\brackets{\begin{array}{c}
				h_x \\
				h_z
				\end{array}}\otimes
				\brackets{\begin{array}{c}
				h_x \\
				h_z
				\end{array}}} \\
				= ~& \frac{1}{2}\trace\braces{
				\brackets{\begin{array}{cc}
				\L & -\L \\
				-\L & \L
				\end{array}}
				\brackets{\begin{array}{cc}
				h_x \otimes h_x & h_x \otimes h_z \\
				h_z \otimes h_x & h_z \otimes h_z
				\end{array}}} \\
				~& + M\trace\braces{
				\brackets{\begin{array}{cc}
				\I & \I \\
				\I & \I
				\end{array}}
				\brackets{\begin{array}{cc}
				h_x \otimes h_x & h_x \otimes h_z \\
				h_z \otimes h_x & h_z \otimes h_z
				\end{array}}} \\
				= ~& \frac{1}{2}\trace\braces{\L\cdot(h_x-h_z)\otimes(h_x-h_z)} \\
				~& + M\trace\braces{(h_x+h_z)\otimes(h_x+h_z)} \\
				= ~& \frac{C}{2}\,\omega''(\abs{x-z})\trace\braces{\W\cdot(h_x-h_z)\otimes(h_x-h_z)} \\
				~& + \frac{C}{2}\,\frac{\omega'(\abs{x-z})}{\abs{x-z}}\trace\braces{(\I-\W)\cdot(h_x-h_z)\otimes(h_x-h_z)} \\
				~& + M\abs{h_x+h_z}^2,
\end{split}
\end{equation*}
and replacing in the second order Taylor's expansion \eqref{TAYLOR00}, we obtain
\begin{multline}\label{TAYLOR0}
				f(x+h_x,z+h_z)-f(x,z) \\
\begin{split}
				= ~& C\,\omega'(\abs{x-z})\prodin{\v}{h_x-h_z}+2M\prodin{x+z}{h_x+h_z} \\
				~& + \frac{C}{2}\,\omega''(\abs{x-z})\trace\braces{\W\cdot(h_x-h_z)\otimes(h_x-h_z)} \\
				~& + \frac{C}{2}\,\frac{\omega'(\abs{x-z})}{\abs{x-z}}\trace\braces{(\I-\W)\cdot(h_x-h_z)\otimes(h_x-h_z)} \\
				~& + M\abs{h_x+h_z}^2 +\mathcal{E}_{x,z}(h_x,h_z).
\end{split}
\end{multline}
Moreover, 
since $\abs{x-z}\leq\omega_1$, 
by the explicit form of the function $\omega$, \eqref{omega},
\begin{equation*}
				\omega'''(t)=-\gamma(\gamma-1)(\gamma-2)\omega_0t^{\gamma-3},
\end{equation*}
for every $\abs{h_x},\abs{h_z}\leq\varepsilon$, by Taylor's theorem, it holds
\begin{equation*}
				\mathcal{E}_{x,z}(h_x,h_z)\leq \gamma(\gamma-1)(2-\gamma)C\omega_0\abs{\brackets{\begin{array}{c}
				h_x \\
				h_z
				\end{array}}}^3(\abs{x-z}-2\varepsilon)^{\gamma-3},
\end{equation*}
whenever $\abs{x-z}>2\varepsilon$. Since $1<\gamma<2$, using the hypothesis $\abs{x-z}>\frac{N}{10}\varepsilon$ and choosing large enough natural number $N\geq 40$ depending on $C$ and $\omega_0$,
\begin{equation}\label{N1}
				N>2^{11/2}\cdot 10\ C\omega_0,
\end{equation}
we can estimate
\begin{equation*}
\begin{split}
				\mathcal{E}_{x,z}(h_x,h_z) ~& \leq 2C\omega_0(2\varepsilon^2)^{3/2}\pare{\frac{\abs{x-z}}{2}}^{\gamma-3} \\
				~& < 2^{11/2}C\omega_0\,\frac{\varepsilon}{\abs{x-z}}\,\abs{x-z}^{\gamma-2}\,\varepsilon^2 \\
				~& < \frac{2^{11/2}\cdot 10\ C\omega_0}{N}\,\abs{x-z}^{\gamma-2}\,\varepsilon^2 \\
				~& \leq \abs{x-z}^{\gamma-2}\,\varepsilon^2.
\end{split}
\end{equation*}
On the other hand, since $\abs{x-z}\leq 1$ and $\gamma-2<0$, we have
\begin{equation*}
				M\abs{h_x+h_z}^2\leq 4M\varepsilon^2\leq 4M\abs{x-z}^{\gamma-2}\,\varepsilon^2,
\end{equation*}
and then the last two terms in \eqref{TAYLOR0} are bounded by
\begin{equation*}
				(4M+1)\,\abs{x-z}^{\gamma-2}\varepsilon^2.
\end{equation*}
Finally, recalling the notation introduced in \eqref{hV}, we obtain \eqref{TAYLOR0000}.
\end{proof}

Now, we utilize expansion \eqref{TAYLOR0000} for obtaining an estimate for the function $F$ defined in \eqref{F-def}.

\begin{lemma}\label{LEMMA-F}
Let $x,z$ as at the beginning of this section and $\abs{\nu_x}=\abs{\nu_z}=1$. Then, there is a pair of matrices $\P_{\nu_x}$ and $\P_{\nu_z}$ such that $\P_{\nu_x}\e_1=\nu_x$, $\P_{\nu_z}\e_1=\nu_z$  and the function $F$ defined in \eqref{F-def} satisfies
\begin{multline}\label{F-expansion}
				F(x,z,\nu_x,\nu_z)-f(x,z) \\
\begin{split}
				\leq ~& C\,\omega'(\abs{x-z})\brackets{\alpha(z)(\nu_x-\nu_z)_V+(\alpha(x)-\alpha(z))(\nu_x)_V}\,\varepsilon \\
				~& + 2M\prodin{x+z}{\alpha(z)(\nu_x+\nu_z)+(\alpha(x)-\alpha(z))\nu_x}\,\varepsilon \\
				~& + \frac{C}{2}\,\omega''(\abs{x-z})\,\alpha(z)(\nu_x-\nu_z)_V^2\,\varepsilon^2 \\
				~& +\frac{C}{2}\,\frac{\omega'(\abs{x-z})}{\abs{x-z}}\Big\{\alpha(z)(\nu_x-\nu_z)_{V^\bot}^2+\beta(x)\abs{\nu_x+\nu_z}^2 \\
				~& \hspace{80pt}+(\alpha(x)-\alpha(z))\brackets{1+(\nu_x)_{V^\bot}^2}\Big\}\varepsilon^2 \\
				~& + (4M+1)\,\abs{x-z}^{\gamma-2}\varepsilon^2.
\end{split}
\end{multline}
\end{lemma}

\begin{proof}
First, replacing $h_x=\varepsilon\,\nu_x$ and $h_z=\varepsilon\,\nu_z$ in \eqref{TAYLOR0000}, we get the following for the $\alpha(z)$-term in \eqref{F-def},
\begin{multline}\label{[I]}
				f(x+\varepsilon\nu_x,z+\varepsilon\nu_z) - f(x,z)\\
\begin{split}
				\leq ~& C\,\omega'(\abs{x-z})(\nu_x-\nu_z)_V\,\varepsilon + 2M\prodin{x+z}{\nu_x+\nu_z}\,\varepsilon \\
				~& + \frac{C}{2}\,\omega''(\abs{x-z})(\nu_x-\nu_z)_V^2\,\varepsilon^2+\frac{C}{2}\,\frac{\omega'(\abs{x-z})}{\abs{x-z}}(\nu_x-\nu_z)_{V^\bot}^2\,\varepsilon^2 \\
				~& + (4M+1)\,\abs{x-z}^{\gamma-2}\varepsilon^2.
\end{split}
\end{multline}
Similarly, for the $\beta(x)$-term,
\begin{multline*}
				f(x+\varepsilon\,\P_{\nu_x}\zeta,z+\varepsilon\,\P_{\nu_z}\zeta) - f(x,z) \\
\begin{split}
				\leq ~&  C\,\omega'(\abs{x-z})\,(\P_{\nu_x}\zeta-\P_{\nu_z}\zeta)_V\,\varepsilon + 2M\prodin{x+z}{\P_{\nu_x}\zeta+\P_{\nu_z}\zeta}\,\varepsilon \\
				~& + \frac{C}{2}\,\omega''(\abs{x-z})(\P_{\nu_x}\zeta-\P_{\nu_z}\zeta)_V^2\,\varepsilon^2 +\frac{C}{2}\,\frac{\omega'(\abs{x-z})}{\abs{x-z}}(\P_{\nu_x}\zeta-\P_{\nu_z}\zeta)_{V^\bot}^2\,\varepsilon^2 \\
				~& + (4M+1)\,\abs{x-z}^{\gamma-2}\varepsilon^2.
\end{split}
\end{multline*}
Integrating with respect to the $(n-1)$-dimensional Lebesgue measure on $\B^{\e_1}$ the first order terms vanish, while for the second order terms, we use the concavity of $\omega$ to estimate $\omega''\leq 0$. Moreover, since we can choose $\P_{\nu_x}$ and $\P_{\nu_z}$ satisfying
\begin{equation*}
				\abs{\P_{\nu_x}\zeta-\P_{\nu_z}\zeta}\leq\abs{\nu_x+\nu_z}
\end{equation*}
for every $\zeta\in\B^{\e_1}$ (see \Cref{matrices}), we get
\begin{multline}\label{[II]}
				\dashint_{\B^{\e_1}}f(x+\varepsilon\,\P_{\nu_x}\zeta,z+\varepsilon\,\P_{\nu_z}\zeta)\ d\Ln(\zeta) -f(x,z) \\
				\leq \frac{C}{2}\,\frac{\omega'(\abs{x-z})}{\abs{x-z}}\,\abs{\nu_x+\nu_z}^2\,\varepsilon^2 + (4M+1)\,\abs{x-z}^{\gamma-2}\varepsilon^2.
\end{multline}
Finally, for the last term in \eqref{F-def},
\begin{multline*}
				f(x+\varepsilon\,\nu_x,z+\varepsilon\,\P_{\nu_z}\zeta) - f(x,z) \\
\begin{split}
				\leq ~&  C\,\omega'(\abs{x-z})\,(\nu_x-\P_{\nu_z}\zeta)_V\,\varepsilon + 2M\prodin{x+z}{\nu_x+\P_{\nu_z}\zeta}\,\varepsilon \\
				~& + \frac{C}{2}\,\omega''(\abs{x-z})(\nu_x-\P_{\nu_z}\zeta)_V^2\,\varepsilon^2 +\frac{C}{2}\,\frac{\omega'(\abs{x-z})}{\abs{x-z}}(\nu_x-\P_{\nu_z}\zeta)_{V^\bot}^2\,\varepsilon^2 \\
				~& + (4M+1)\,\abs{x-z}^{\gamma-2}\varepsilon^2.
\end{split}
\end{multline*}
Due to symmetry, the first order terms containing $\zeta$ cancel out after integration over $\B^{\e_1}$, while for the second order terms, we use the rough estimate $\omega''\leq 0$. For the remaining term, we develop $(\nu_x-\P_{\nu_z}\zeta)_{V^\bot}^2$ using notation \eqref{hV},
\begin{equation*}
\begin{split}
				(\nu_x-\P_{\nu_z}\zeta)_{V^\bot}^2
				~& =\abs{\nu_x-\P_{\nu_z}\zeta}^2-(\nu_x-\P_{\nu_z}\zeta)_V^2 \\
				~& =\abs{\nu_x}^2-(\nu_x)_V^2+\abs{\P_{\nu_z}\zeta}^2-(\P_{\nu_z}\zeta)_V^2-2\brackets{\prodin{\nu_x}{\P_{\nu_z}\zeta}-(\nu_x)_V(\P_{\nu_z}\zeta)_V} \\
				~& =(\nu_x)_{V^\bot}^2+(\P_{\nu_z}\zeta)_{V^\bot}^2-2\prodin{\nu_x-(\nu_x)_V\v}{\P_{\nu_z}\zeta}.
\end{split}
\end{equation*}
Note that, again by symmetry, the last term vanishes after integration and, since $(\P_{\nu_z}\zeta)_{V^\bot}^2\leq 1$ for any $\abs{\zeta}\leq 1$, we get
\begin{equation*}
				\dashint_{\B^{\e_1}}(\nu_x-\P_{\nu_z}\zeta)_{V^\bot}^2\ d\Ln(\zeta) \leq (\nu_x)_{V^\bot}^2+1.
\end{equation*}
Therefore,
\begin{multline}\label{[III]}
				\dashint_{\B^{\e_1}}f(x+\varepsilon\,\nu_x,z+\varepsilon\,\P_{\nu_z}\zeta)d \Ln(\zeta) - f(x,z) \\
\begin{split}
				\leq ~&  C\,\omega'(\abs{x-z})(\nu_x)_V\,\varepsilon + 2M\prodin{x+z}{\nu_x}\,\varepsilon
				+\frac{C}{2}\,\frac{\omega'(\abs{x-z})}{\abs{x-z}}\brackets{1+(\nu_x)_{V^\bot}^2}\,\varepsilon^2 \\
				~& + (4M+1)\,\abs{x-z}^{\gamma-2}\varepsilon^2.
\end{split}
\end{multline}

Then, replacing each of the terms \eqref{[I]}, \eqref{[II]} and \eqref{[III]} in the formula for $F$ \eqref{F-def}, we get \eqref{F-expansion} and finish the proof.
\end{proof}
\ \\

Now we are in position to explain the game intuition behind our argument of the proof of \Cref{PROPOSITION}. Recall that the crucial point in proving the key \Cref{PROPOSITION} is that, given the choices $\nu_x,\nu_z$ of our opponent, we need to find appropriate vectors $|\widetilde\nu_x|=|\widetilde\nu_z|=1$ so that \eqref{F+F-2f<0} holds. Before moving on to details, we will give intuition for our strategy. We mentioned already in Section \ref{sec:intuition} that the strategy of always pulling the points directly closer to each other does not provide the desired result in general. Hence, our response will depend on the opponents choice. If the opponent chooses to pull the points almost as far from each other that is possible, our response is to pull directly to the opposite direction by choosing $\widetilde\nu_x=-\nu_x$ and $\widetilde\nu_z=-\nu_z$. Otherwise, we just pull the points directly towards each other by choosing $\widetilde\nu_x=-\v$ and $\widetilde\nu_z=\v$, where $\v=\frac{x-z}{|x-z|}$. The way of making the distinction is to consider the projection $(\nu_x-\nu_z)_V$ and fix the threshold $\Theta=\Theta(x,z)$. As we will see 
in \eqref{introcite2}, the particularities of our comparison function $f_1$ make it necessary to require the function $\alpha(\cdot)$ to be  H\"older continuous (with a H\"older exponent $s>0$) and to choose the threshold depending both on the distance of the points as well as the H\"older exponent of $\alpha$, $\Theta(x,z)=\abs{x-z}^{s}\in(0,1]$.

Now we continue with the proof of the key \Cref{PROPOSITION}.


\subsection{Case 1. $(\nu_x-\nu_z)_V^2\geq 4-\Theta$.}\label{Case1}
This is the case where the opponent plays pulling the points almost in the opposite direction. In this case, as a response to the choices of our opponent, we select $\widetilde\nu_x=-\nu_x$ and $\widetilde\nu_z=-\nu_z$. Replacing these in the right hand side of \eqref{F+F-2f<0} and recalling the expansion \eqref{F-expansion}, it turns out that the first order terms cancel out and we get
\begin{multline*}
				F(x,z,\nu_x,\nu_z)+F(x,z,-\nu_x,-\nu_z)-2f(x,z) \\
\begin{split}
				\leq ~& C\,\omega''(\abs{x-z})\,\alpha(z)(\nu_x-\nu_z)_V^2\,\varepsilon^2 \\
				~& +C\,\frac{\omega'(\abs{x-z})}{\abs{x-z}}\Big\{\alpha(z)(\nu_x-\nu_z)_{V^\bot}^2+\beta(x)\abs{\nu_x+\nu_z}^2 \\
				~& \hspace{80pt}+(\alpha(x)-\alpha(z))\brackets{1+(\nu_x)_{V^\bot}^2}\Big\}\varepsilon^2 \\
				~& + 2(4M+1)\,\abs{x-z}^{\gamma-2}\varepsilon^2.
\end{split}
\end{multline*}
Recalling the properties of the function $\omega$, $\frac{1}{2}\leq\omega'\leq 1$ and $\omega''\leq 0$, we obtain
\begin{multline}\label{ineq00}
				F(x,z,\nu_x,\nu_z)+F(x,z,-\nu_x,-\nu_z)-2f(x,z) \\
\begin{split}
				\leq ~& 3\,\alpha_\textrm{min}\,C\,\omega''(\abs{x-z})\,\varepsilon^2 \\
				~& +C\,\frac{1}{\abs{x-z}}\Big\{\alpha(z)(\nu_x-\nu_z)_{V^\bot}^2+\beta(x)\abs{\nu_x+\nu_z}^2 \\
				~& \hspace{60pt}+(\alpha(x)-\alpha(z))\brackets{1+(\nu_x)_{V^\bot}^2}\Big\}\varepsilon^2 \\
				~& + 2(4M+1)\,\abs{x-z}^{\gamma-2}\varepsilon^2,
\end{split}
\end{multline}
where the inequality $(\nu_x-\nu_z)_V^2\geq 4-\Theta\geq 3$ has been used together with $\alpha(z)\geq\alpha_\textrm{min}$. Thus, we need to obtain estimates for $(\nu_x-\nu_z)_{V^\bot}^2$, $\abs{\nu_x+\nu_z}^2$ and $(\nu_x)_{V^\bot}^2$. The first one follows directly from the hypothesis and Pythagorean theorem,
\begin{equation*}
				(\nu_x-\nu_z)_{V^\bot}^2
				= \abs{\nu_x-\nu_z}^2-(\nu_x-\nu_z)_V^2
				\leq 4-(\nu_x-\nu_z)_V^2
				\leq \Theta,
\end{equation*}
while for the second one we recall the parallelogram law,
\begin{equation*}
				\abs{\nu_x+\nu_z}^2=4-\abs{\nu_x-\nu_z}^2\leq 4-(\nu_x-\nu_z)_V^2\leq\Theta.
\end{equation*}
On the other hand,
\begin{equation*}
				(\nu_x)_V^2\geq (1-\sqrt{4-\Theta})^2,
\end{equation*}
and since
\begin{equation}\label{ineq-root}
				\sqrt{4-\Theta}=\sqrt{\pare{2-\frac{\Theta}{4}}^2-\frac{\Theta^2}{16}}\leq 2-\frac{\Theta}{4},
\end{equation}
then
\begin{equation*}
				(1-\sqrt{4-\Theta})^2=5-\Theta-2\sqrt{4-\Theta}\geq 1-\frac{\Theta}{2}.
\end{equation*}
Consequently, we obtain the following estimate for $(\nu_x)_{V^\bot}^2$:
\begin{equation*}
				(\nu_x)_{V^\bot}^2=\abs{\nu_x}^2-(\nu_x)_V^2=1-(\nu_x)_V^2\leq \dfrac{\Theta}{2}.
\end{equation*}
Thus, recalling that $\beta(x)=1-\alpha(x)$ and $\alpha(x)\geq\alpha(z)$,
\begin{multline*}
				\alpha(z)(\nu_x-\nu_z)_{V^\bot}^2+\beta(x)\abs{\nu_x+\nu_z}^2+(\alpha(x)-\alpha(z))\brackets{1+(\nu_x)_{V^\bot}^2} \\
\begin{split}
				~& \leq (\alpha(z)+\beta(x))\Theta + (\alpha(x)-\alpha(z))\pare{1+\frac{\Theta}{2}} \\
				~& = \Theta+(\alpha(x)-\alpha(z))\pare{1-\frac{\Theta}{2}} \\
				~& \leq \Theta+\alpha(x)-\alpha(z).
\end{split}
\end{multline*}
Then, replacing this estimate in \eqref{ineq00}, we get
\begin{multline}\label{introcite}
				F(x,z,\nu_x,\nu_z)+F(x,z,-\nu_x,-\nu_z)-2f(x,z) \\
				\leq \braces{C \brackets{3\,\alpha_\textrm{min}\,\omega''(\abs{x-z}) + \frac{\Theta}{\abs{x-z}}+\frac{\alpha(x)-\alpha(z)}{\abs{x-z}}} + 2(4M+1)\abs{x-z}^{\gamma-2}}\varepsilon^2.
\end{multline}
Then, by inserting \eqref{omega''} with $\gamma=1+s$, using the precise choice of the threshold $\Theta=\abs{x-z}^s$ and the H\"older estimate \eqref{alpha-Holder} for the function $\alpha(\cdot)$, we obtain
\begin{multline}\label{introcite2}
				3\,\alpha_\mathrm{min}\,\omega''(\abs{x-z}) + \frac{\Theta}{\abs{x-z}}+\frac{\alpha(x)-\alpha(z)}{\abs{x-z}} \\
				\leq \pare{-3\,\alpha_\mathrm{min}s(1+s)\omega_0+1+C_\alpha}\abs{x-z}^{s-1}.
\end{multline}
Then, fixing
\begin{equation}\label{omega0}
				\omega_0\geq\frac{C_\alpha+2}{3\,\alpha_\mathrm{min}s(1+s)},
\end{equation}
and replacing these in \eqref{introcite} we get
\begin{equation*}
				F(x,z,\nu_x,\nu_z)+F(x,z,-\nu_x,-\nu_z)-2f(x,z)
				\leq \braces{2(4M+1)-C}\abs{x-z}^{s-1}\,\varepsilon^2.
\end{equation*}
Choosing large enough
\begin{equation}\label{C-case1}
				C>2(4M+1),
\end{equation}
where $M$ is the constant fixed in \eqref{M}, the negativeness of the previous expression is ensured and \eqref{F+F-2f<0} is proven.


\subsection{Case 2. $(\nu_x-\nu_z)_V^2\leq 4-\Theta$.}\label{Case2}
In that case, by \eqref{ineq-root},
\begin{equation}\label{eq-case2}
				(\nu_x-\nu_z)_V \leq 2-\frac{\Theta}{4}.
\end{equation}
As we noted before, this corresponds to the case where the opponent is not playing near optimality. Then, as a response to her choices, 
we choose $\widetilde\nu_x=-\v$ and $\widetilde\nu_z=\v$. Then, replacing in \eqref{F-expansion} and estimating the $\omega''$-term directly by zero,
\begin{equation*}
\begin{split}
				F(x,z,-\v,\v)-f(x,z)
				\leq ~& C\,\omega'(\abs{x-z})\brackets{-2\alpha(z)-(\alpha(x)-\alpha(z))}\,\varepsilon \\
				~& + 2M\prodin{x+z}{-(\alpha(x)-\alpha(z))\v}\,\varepsilon \\
				~& +\frac{C}{2}\,\frac{\omega'(\abs{x-z})}{\abs{x-z}}(\alpha(x)-\alpha(z))\,\varepsilon^2 \\
				~& + (4M+1)\abs{x-z}^{\gamma-2}\varepsilon^2.
\end{split}
\end{equation*}
Using this and \eqref{F-expansion} in \eqref{F+F-2f<0}, together with the rough estimate $\omega''\leq 0$, we have
\begin{multline*}
				F(x,z,\nu_x,\nu_z)+F(x,z,\widetilde\nu_x,\widetilde\nu_z)-2f(x,z) \\
\begin{split}
				\leq ~& C\,\omega'(\abs{x-z})\braces{\alpha(z)\brackets{(\nu_x-\nu_z)_V-2}+(\alpha(x)-\alpha(z))\brackets{(\nu_x)_V-1}}\,\varepsilon \\
				~& + 2M\prodin{x+z}{\alpha(z)(\nu_x+\nu_z)+(\alpha(x)-\alpha(z))(\nu_x-\v)}\,\varepsilon \\
				~& +\frac{C}{2}\,\frac{\omega'(\abs{x-z})}{\abs{x-z}}\Big\{\alpha(z)(\nu_x-\nu_z)_{V^\bot}^2+\beta(x)\abs{\nu_x+\nu_z}^2 \\
				~& \hspace{80pt}+(\alpha(x)-\alpha(z))\brackets{2+(\nu_x)_{V^\bot}^2}\Big\}\,\varepsilon^2 \\
				~& + 2(4M+1)\abs{x-z}^{\gamma-2}\varepsilon^2.
\end{split}
\end{multline*}
Now, recalling that $(\nu_x)_{V^\bot}^2\leq 1$, $(\nu_x)_V\leq 1$, $\abs{x-z}>\frac{N}{10}\varepsilon$ and rearranging terms,
\begin{multline}\label{CASE2}
				F(x,z,\nu_x,\nu_z)+F(x,z,\widetilde\nu_x,\widetilde\nu_z)-2f(x,z) \\
\begin{split}
				\leq ~& 2M\prodin{x+z}{\alpha(z)(\nu_x+\nu_z)+(\alpha(x)-\alpha(z))(\nu_x-\v)}\,\varepsilon \\
				~& + C\,\omega'(\abs{x-z})\bigg\{\alpha(z)\brackets{(\nu_x-\nu_z)_V-2} \\
				~& \hspace{75pt} + \frac{5}{N}\Big[\alpha(z)(\nu_x-\nu_z)_{V^\bot}^2+\beta(x)\abs{\nu_x+\nu_z}^2 \\
				~& \hspace{105pt} +3(\alpha(x)-\alpha(z))\Big]\bigg\}\,\varepsilon \\
				~& + \frac{20}{N}(4M+1)\abs{x-z}^{\gamma-1}\varepsilon.
\end{split}
\end{multline}
Let us estimate the first term in \eqref{CASE2}. We have
\begin{equation*}
				2M\prodin{x+z}{\alpha(z)(\nu_x+\nu_z)+(\alpha(x)-\alpha(z))(\nu_x-\v)}\,\varepsilon \leq 4M\abs{x+z}\,\varepsilon.
\end{equation*}
Now we focus on the quantity $\abs{x+z}$. Since $x$ and $z$ are points in $B_r$ satisfying \eqref{contra2}, then
\begin{equation*}
\begin{split}
				0 ~& < u(x)-u(z)-f(x,z) \\
				~& =u(x)-u(z)-C\omega(\abs{x-z})-M\abs{x+z}^2 \\
				~& \leq u(x)-u(z)-M\abs{x+z}^2,
\end{split}
\end{equation*}
where we have taken into account the explicit form of the function $f$ in this section, \eqref{f} and the fact that $\omega$ is a positive function. We can rearrange terms and take the square root to get
\begin{equation}\label{estimate}
				\abs{x+z} < \frac{1}{\sqrt{M}} \ \brackets{u(x)-u(z)}^{1/2}.
\end{equation}
At this point, we recall a previous local regularity result from \cite{ARR-HEI-PAR} stating that a function $u=u_\varepsilon$ satisfying \eqref{DPP} is asymptotically Hölder continuous for some exponent $\delta\in(0,1)$, that is,
\begin{equation*}
				\abs{u(x)-u(z)}\leq C_u\big(\abs{x-z}^\delta+\varepsilon^\delta\big),
\end{equation*}
for some constant $C_u>0$ depending on $\alpha_\textrm{min}$, $\alpha_\textrm{max}$, $n$, $r$, $\sup_{B_{2r}}u$ and $\delta$. In particular, using the inequality
\begin{equation*}
				a+b< a\pare{1+\frac{b}{2a}}^2 \quad \quad (a,b>0),
\end{equation*}
we obtain
\begin{equation*}
				\abs{u(x)-u(z)}< C_u\abs{x-z}^\delta\brackets{1+\frac{1}{2}\pare{\frac{\varepsilon}{\abs{x-z}}}^\delta}^2.
\end{equation*}
Now, replacing in \eqref{estimate},
\begin{equation*}
\begin{split}
				\abs{x+z} < ~& \sqrt{\frac{C_u}{M}} \ \abs{x-z}^{\delta/2} \ \brackets{1+\frac{1}{2}\pare{\frac{\varepsilon}{\abs{x-z}}}^\delta} \\
				< ~&  \sqrt{\frac{C_u}{M}} \ \abs{x-z}^{\delta/2} \ \brackets{1+\frac{1}{2}\left(\frac{10}{N}\right)^\delta} \\
				< ~&  \frac{3}{2}\,\sqrt{\frac{C_u}{M}} \ \abs{x-z}^{\delta/2},
\end{split}
\end{equation*}
where in the second inequality we have recalled that $\abs{x-z}>\frac{N}{10}\varepsilon$ and the last inequality follows by choosing large enough $N\in\N$ ($N\geq 10$). Thus, the first term in \eqref{CASE2} is bounded by
\begin{equation*}
				 4M\abs{x+z}\,\varepsilon \leq 6\sqrt{MC_u}\ \abs{x-z}^{\delta/2}\,\varepsilon.
\end{equation*}
Then, replacing this and the Hölder regularity estimate for $\alpha$ in \eqref{CASE2} we get
\begin{multline}\label{CASE2b}
				F(x,z,\nu_x,\nu_z)+F(x,z,\widetilde\nu_x,\widetilde\nu_z)-2f(x,z) \\
\begin{split}
				\leq ~& 6\sqrt{MC_u}\ \abs{x-z}^{\delta/2}\,\varepsilon + \frac{20}{N}(4M+1)\abs{x-z}^{\gamma-1}\varepsilon \\
				~& + C\,\omega'(\abs{x-z})\bigg\{\alpha(z)\brackets{(\nu_x-\nu_z)_V-2} \\
				~& \hspace{75pt} + \frac{5}{N}\Big[\alpha(z)(\nu_x-\nu_z)_{V^\bot}^2+\beta(x)\abs{\nu_x+\nu_z}^2 \\
				~& \hspace{105pt} +3C_\alpha\abs{x-z}^s\Big]\bigg\}\,\varepsilon.
\end{split}
\end{multline}
Thus, we need to estimate the terms in braces of the above inequality. One special case happens when $(\nu_x-\nu_z)_{V^\bot}^2\leq \Theta$. Then, the rest of the terms can be easily estimated by using the hypothesis \eqref{eq-case2} and the desired result follows. However, we don't have any control on the size of this term and, for that reason, we need to define a new variable $\vartheta\in\brackets{1,\frac{4}{\Theta}}$ as follows:
\begin{equation}\label{vartheta}
				\vartheta:\,=\vartheta(x,z)=\begin{cases}
				\dfrac{1}{\Theta}(\nu_x-\nu_z)_{V^\bot}^2 & \mbox{ if } (\nu_x-\nu_z)_{V^\bot}^2>\Theta, \\ & \\
				1 & \mbox{ otherwise. }
				\end{cases}
\end{equation}
When $\vartheta>1$, we have
\begin{equation*}
				(\nu_x-\nu_z)_V\leq\sqrt{4-\vartheta\Theta}\leq 2-\frac{\vartheta\Theta}{4}.
\end{equation*}
Note that, by \eqref{eq-case2}, this inequality also holds when $\vartheta=1$. Thus,
\begin{equation}\label{ineq nu-vartheta}
				2-(\nu_x-\nu_z)_V\geq\frac{\vartheta\Theta}{4}.
\end{equation}
Therefore, by \eqref{vartheta} and \eqref{ineq nu-vartheta},
\begin{equation}\label{eq-0001}
				(\nu_x-\nu_z)^2_{V^\bot}\leq\vartheta\Theta\leq 4\brackets{2-(\nu_x-\nu_z)_V}.
\end{equation}
For the second term in brackets, using the parallelogram law we get
\begin{equation}\label{eq-0002}
\begin{split}
				\abs{\nu_x+\nu_z}^2
				~& = 4-\abs{\nu_x-\nu_z}^2 \\
				~& \leq 4-(\nu_x-\nu_z)_V^2 \\
				~& =\brackets{2+(\nu_x-\nu_z)_V}\brackets{2-(\nu_x-\nu_z)_V} \\
				~& < 4\brackets{2-(\nu_x-\nu_z)_V},
\end{split}
\end{equation}
and, since $\Theta=\abs{x-z}^s$ and $\vartheta\geq 1$,
\begin{equation}\label{eq-0003}
				3C_\alpha\abs{x-z}^s\leq 3C_\alpha\vartheta\Theta\leq 4\cdot 3C_\alpha\brackets{2-(\nu_x-\nu_z)_V}.
\end{equation}
Then, combining \eqref{eq-0001}, \eqref{eq-0002} and \eqref{eq-0003}, and since $\alpha(x)\geq\alpha(z)$,
\begin{multline*}
				\alpha(z)(\nu_x-\nu_z)_{V^\bot}^2+\beta(x)\abs{\nu_x+\nu_z}^2 +3C_\alpha\abs{x-z}^s \\
				\leq 4(3C_\alpha+1)\brackets{2-(\nu_x-\nu_z)_V}.
\end{multline*}
Therefore, replacing in \eqref{CASE2b},
\begin{multline*}
				F(x,z,\nu_x,\nu_z)+F(x,z,\widetilde\nu_x,\widetilde\nu_z)-2f(x,z) \\
\begin{split}
				\leq ~& 6\sqrt{MC_u}\ \abs{x-z}^{\delta/2}\,\varepsilon + \frac{20}{N}(4M+1)\abs{x-z}^{\gamma-1}\varepsilon \\
				~& + C\,\omega'(\abs{x-z})\brackets{2-(\nu_x-\nu_z)_V}\braces{-\alpha(z) + \frac{20}{N}(3C_\alpha+1)}\,\varepsilon.
\end{split}
\end{multline*}
Choosing $N\in\N$ such that
\begin{equation}\label{N2}
				N>40\,\frac{3C_\alpha+1}{\alpha_\mathrm{min}},
\end{equation}
and since $\alpha(z)\geq\alpha_\mathrm{min}$, we get
\begin{equation*}
				-\alpha(z) + \frac{20}{N}(3C_\alpha+1) \leq -\frac{\alpha_\mathrm{min}}{2}<0.
\end{equation*}
Finally, recalling \eqref{eq-case2}, $\omega'\geq\frac{1}{2}$, $\Theta=\abs{x-z}^s$ and $s=\gamma-1=\delta/2$, we obtain
\begin{multline*}
				F(x,z,\nu_x,\nu_z)+F(x,z,\widetilde\nu_x,\widetilde\nu_z)-2f(x,z) \\
				\leq \brackets{6\sqrt{MC_u}\  + \frac{\alpha_\mathrm{min}}{2}\cdot\pare{\frac{4M+1}{3C_\alpha+1}-\frac{C}{8}}}\abs{x-z}^s\,\varepsilon.
\end{multline*}
Choosing large enough
\begin{equation}\label{C-case2}
				C>8\pare{\frac{4M+1}{3C_\alpha+1}+\frac{12}{\alpha_\mathrm{min}}\sqrt{MC_u}}
\end{equation}
depending on $M$, $C_\alpha$, $\alpha_\mathrm{min}$ and $C_u$, we ensure that \eqref{F+F-2f<0} holds.


\section{Proof of \Cref{PROPOSITION}. case $\abs{x-z}\leq \frac{N}{10}\varepsilon$}\label{remarks}

In the previous section, we proved \Cref{PROPOSITION} in the case $\abs{x-z}>\frac{N}{10}\varepsilon$. The other case $\abs{x-z}\leq \frac{N}{10}\varepsilon$ is similar to \cite{ARR-HEI-PAR}. In Section \ref{Comp} we briefly commented that in this case we need an annular step function $f_2\lesssim \eps$.
Recalling \eqref{TAYLOR0000} and for large enough
\begin{equation}\label{C-Section4-1}
				C>8Mr+1,
\end{equation}
we obtain the following rough estimate for $f_1$,
\begin{multline}\label{rough}
                f_1(x+h_x,z+h_z)-f_1(x,z) \\
\begin{split}
				~& \leq C\,\omega'(\abs{x-z})(h_x-h_z)_V+2M\prodin{x+z}{h_x+h_z} + \abs{x-z}^{\gamma-1}\,\varepsilon \\
				~& \leq (2C+4M\abs{x+z}+1)\varepsilon \\
				~& < 3C\varepsilon.
\end{split}
\end{multline}
Replacing $f=f_1-f_2$ in \eqref{F-def}, we decompose $F=G_1-G_2$, where
\begin{equation*}
\begin{split}
				G_i(x,z,\nu_x,\nu_z):\,= ~& F(f_i,x,z,\nu_x,\nu_z,\varepsilon) = \alpha(z)f_i(x+\varepsilon\nu_x,z+\varepsilon\nu_z) \\
				~& +\beta(x)\dashint_{\B^{\e_1}}f_i(x+\varepsilon\,\P_{\nu_x}\zeta,z+\varepsilon\,\P_{\nu_z}\zeta)\ d\mathcal{L}^{n-1}(\zeta) \\
				~& +(\alpha(x)-\alpha(z))\dashint_{\B^{\e_1}}f_i(x+\varepsilon\nu_x,z+\varepsilon\,\P_{\nu_z}\zeta)\ d\mathcal{L}^{n-1}(\zeta),
\end{split}
\end{equation*}
for $i=1,2$. Then, by \eqref{rough}, we can estimate
\begin{equation}\label{eq:tilde g}
				\sup_{\nu_x,\nu_z}G_1(x,z,\nu_x,\nu_z)\leq  f_1(x,z)+ 3C\varepsilon.
\end{equation}
Together with $f_2\geq 0$, these estimates yield
\begin{equation}\label{eq:tilde f}
				\sup_{\nu_x,\nu_z}F(x,z,\nu_x,\nu_z)\leq  f_1(x,z)+ 3C\varepsilon.
\end{equation}
Recalling the definition of the step annular function \eqref{f2}, fix $i\in \{0,1,2,\ldots,N\}$ such that $(x,z)\in A_i$ and choose $\abs{\widetilde\nu_x}=\abs{\widetilde\nu_z}=1$ such that $(x+\varepsilon\widetilde\nu_x,z+\varepsilon\widetilde\nu_z)\in A_{i-1}$. Then for $C>1$ large enough such that
\begin{equation}\label{C-Section4-2}
				\alpha_\mathrm{min}C^2-2>7C,
\end{equation}
we can estimate
\begin{equation*}
\begin{split}
				\sup_{\nu_x,\nu_z}G_2(x,z,\nu_x,\nu_z) \geq ~& G_2(x,z,\widetilde\nu_x,\widetilde\nu_z) \\
				\geq ~& \alpha(z)f_2(x+\varepsilon\widetilde\nu_x,z+\varepsilon\widetilde\nu_z) \\
				\geq ~& \alpha_\mathrm{min}f_2(x+\varepsilon\widetilde\nu_x,z+\varepsilon\widetilde\nu_z) \\
				= ~& \alpha_\mathrm{min}C^{2(N-i+1)}\varepsilon \\
				= ~& \alpha_\mathrm{min}C^2C^{2(N-i)}\varepsilon-2C^{2(N-i)}\varepsilon+2f_2(x,z) \\
				= ~& \alpha_\mathrm{min}\pare{C^2-\frac{2}{\alpha_\mathrm{min}}}C^{2(N-i)}\varepsilon+2f_2(x,z) \\
				> ~& 7C\varepsilon+2f_2(x,z),
\end{split}
\end{equation*}
where we use $f_2\geq 0$ in the second inequality and $\alpha_{\text{min}}>0$ in the last inequality. Therefore, by $f=f_1-f_2$ and \eqref{eq:tilde g} it holds
\begin{equation*}
\begin{split}
				\inf_{\nu_x,\nu_z}F(x,z,\nu_x,\nu_z) \leq  ~& \sup_{\nu_x,\nu_z} G_1(x,z,\nu_x,\nu_z)-\sup_{\nu_x,\nu_z} G_2(x,z,\nu_x,\nu_z)\\
				\leq ~& f_1(x,z)-2f_2(x,z)-4C\varepsilon.
\end{split}
\end{equation*}
Combining this inequality with \eqref{eq:tilde f}, we get
\begin{equation*}
				\sup_{\nu_x,\nu_z}F(x,z,\nu_x,\nu_z)+\inf_{\nu_x,\nu_z}F(x,z,\nu_x,\nu_z) < 2f(x,z)-C\varepsilon.
\end{equation*}
Letting large enough $C$, we get \eqref{F+F-2f<0}, and this proves \Cref{PROPOSITION} in the case $\abs{x-z} \leq \frac{N}{10}\varepsilon$.


\section{An alternative formulation in the case $2<p(x)<\infty$}\label{2<p<infty}

As we noted at the beginning of this work, the authors in \cite{ARR-HEI-PAR} showed that the solutions $u_\varepsilon$ of the DPP \eqref{DPP} converge uniformly as $\varepsilon\to 0$ to a viscosity solution of the normalized $p(x)$-Laplace equation
\begin{equation*}
				\Delta^N_{p(x)}\, u(x)=\Delta u(x)+(p(x)-2)\Delta^N_\infty\, u(x)=0,
\end{equation*}
provided that $p:\Omega\to(1,\infty]$ is a continuous function. In this section we consider a different DPP whose solutions are asymptotically related in the same way to the normalized $p(x)$-Laplace equation when $p(x)>2$ for all $x\in\Omega$. Given $\Omega\subset\R^n$ a bounded domain and small enough $\varepsilon>0$, let $u=u_\varepsilon:\Omega\rightarrow\R$ be a function satisfying the DPP 
\begin{equation}\label{Classic-DPP}
				u(x) = \alpha(x)\braces{\frac{1}{2}\sup_{B_\varepsilon(x)}u + \frac{1}{2}\inf_{B_\varepsilon(x)}u} +\beta(x)\dashint_{B_\varepsilon(x)} \hspace{-5pt} u
\end{equation}
for $x\in\Omega$, where $\alpha:\Omega\to(0,1]$ and $\beta:\Omega\to[0,1)$ are continuous probability functions depending on $p$ and defined as follows:
\begin{equation*}
				\alpha(x):\,=\frac{p(x)-2}{n+p(x)} \quad \quad \mbox{ and } \quad \quad \beta(x):\,=\frac{n+2}{n+p(x)}.
\end{equation*}

As it happens with \eqref{DPP}, the DPP \eqref{Classic-DPP} is related to a slightly different tug-of-war game, compared to the DPP \eqref{DPP}. Indeed, the main difference between this game and the previous one is that, in this case, the random noise can displace the token to any point in the $n$-dimensional ball $B_\varepsilon(x)$, instead of moving it to a random point in the orthogonal $(n-1)$-dimensional ball $B_\varepsilon^{\nu}(x)$, where $\nu$ is the direction chosen by the winner of the toss. That is, the possible random displacement of the token in a single step is not affected by the choices of the players. For more details, see \cite{manfredipr12} where this game is described for fixed $\alpha$ and $\beta$.
\\

In a previous result (see \cite[Section 5]{luirop}), it was shown that, for given bounded domain $\Omega\subset\R^n$ and $B_{2r}(x_0)\subset\Omega$, a solution $u=u_\varepsilon$ of \eqref{Classic-DPP} satisfies
\begin{equation}\label{Asymptotic-Holder-Regularity-LP}
				\abs{u(x)-u(z)}\leq C_u\pare{\abs{x-z}^\delta+\varepsilon^\delta} \quad\quad \mbox{ where }\ x,z\in B_r(x_0),
\end{equation}
for some exponent $\delta\in(0,1)$.
\\

As in the case studied in previous sections, provided that the function $p$ is Hölder continuous, that is,
\begin{equation*}
				\abs{p(x)-p(z)}\leq C_p\abs{x-z}^s,
\end{equation*}
for every $x,y\in\Omega$ and some $C_p>0$ and $s\in(0,1)$, the asymptotic estimate \eqref{Asymptotic-Holder-Regularity-LP} can be shown with $\delta=1$.

\begin{theorem}\label{MAIN-THM-Classic}
Let $\Omega\subset\R^n$ be a bounded domain and $B_{2r}(x_0)\subset\Omega$ for some $r>0$. Then, for a solution $u=u_\varepsilon$ of \eqref{Classic-DPP} it holds
\begin{equation*}
				\abs{u(x)-u(z)}\leq C\pare{\abs{x-z}+\varepsilon} \quad\quad \mbox{ when }\ x,z\in B_r(x_0),
\end{equation*}
for some constant $C>0$ depending on $p_\textrm{min}$, $C_p$, $n$, $r$ and $\sup_{B_{2r}}u$.
\end{theorem}

We show that the asymptotic regularity result for solutions $u$ of \eqref{Classic-DPP} stated in the previous theorem can be directly derived from the arguments in \Cref{Section2,remarks}. Let us rewrite \eqref{Classic-DPP} using the midrange notation introduced at the beginning of this article. Since the $\beta(x)$-term of the DPP does not depend on any parameter, \eqref{Classic-DPP} can be written as
\begin{equation*}
				u(x) = \midr_{h\in\B}\A_\varepsilon u(x,h),
\end{equation*}
where $\B=B_1(0)$ stands for the unitary ball centered at the origin and 
\begin{equation*}
				\A_\varepsilon u(x,h)=\alpha(x)u(x+\varepsilon h)+\beta(x)\dashint_\B u(x+\varepsilon\,\zeta)\ d\zeta,
\end{equation*}
which is a similar version of \eqref{DPP2} and \eqref{Au2}, respectively. Thus, given $x,z\in B_r$ and $h_x,h_z\in\B$ and assuming without any loss of generality that $\alpha(x)\geq\alpha(z)$, we analogously get
\begin{equation*}
\begin{split}
				\A_\varepsilon u(x,h_x) - \A_\varepsilon u(z,h_z)
				= ~& \alpha(z)\brackets{u(x+\varepsilon h_x)-u(z+\varepsilon h_z)} \\
				~& +\beta(x)\dashint_\B\brackets{u(x+\varepsilon\,\zeta)-u(z+\varepsilon\,\zeta)}\ d\zeta \\
				~& +(\alpha(x)-\alpha(z))\dashint_\B\brackets{u(x+\varepsilon h_x)-u(z+\varepsilon\,\zeta)}\ d\zeta.
\end{split}
\end{equation*}
Proceeding by contradiction in the same way as in \Cref{background} (see \Cref{LEMMA-1}), we will end up defining a function $F$ as follows,
\begin{equation}\label{Classic-F-def} 
\begin{split}
				F(x,z,h_x,h_z):\,= ~& F(f,x,z,h_x,h_z,\varepsilon):\,= \alpha(z)f(x+\varepsilon h_x,z+\varepsilon h_z) \\
				~& +\beta(x)\dashint_\B f(x+\varepsilon\,\zeta,z+\varepsilon\,\zeta)\ d\zeta \\
				~& +(\alpha(x)-\alpha(z))\dashint_\B f(x+\varepsilon h_x,z+\varepsilon\,\zeta)\ d\zeta,
\end{split}
\end{equation}
for $h_x,h_z\in\B$, and we show the following expansion for $F$:

\begin{lemma}\label{LEMMA-F-Classic}
Let $h_x,h_z\in\B$. Then, for $\abs{x-z}>>\varepsilon$, the function $F$ defined in \eqref{Classic-F-def} satisfies
\begin{multline}\label{Classic-F-expansion}
				F(x,z,h_x,h_z)-f(x,z) \\
\begin{split}
				\leq ~& C\,\omega'(\abs{x-z})\brackets{\alpha(z)(h_x-h_z)_V+(\alpha(x)-\alpha(z))(h_x)_V}\,\varepsilon \\
				~& + 2M\prodin{x+z}{\alpha(z)(h_x+h_z)+(\alpha(x)-\alpha(z))h_x}\,\varepsilon \\
				~& + \frac{C}{2}\,\omega''(\abs{x-z})\,\alpha(z)(h_x-h_z)_V^2\,\varepsilon^2 \\
				~& +\frac{C}{2}\,\frac{\omega'(\abs{x-z})}{\abs{x-z}}\Big\{\alpha(z)(h_x-h_z)_{V^\bot}^2+(\alpha(x)-\alpha(z))\brackets{1+(h_x)_{V^\bot}^2}\Big\}\varepsilon^2 \\
				~& + (4M+1)\,\abs{x-z}^{\gamma-2}\varepsilon^2.
\end{split}
\end{multline}
\end{lemma}

\begin{proof}
The $\alpha(z)$-term in \eqref{Classic-F-def} follows directly from \eqref{TAYLOR0000},
\begin{multline}\label{Classic-[I]}
				f(x+\varepsilon h_x,z+\varepsilon h_z) - f(x,z)\\
\begin{split}
				\leq ~& C\,\omega'(\abs{x-z})(h_x-h_z)_V\,\varepsilon + 2M\prodin{x+z}{h_x+h_z}\,\varepsilon \\
				~& + \frac{C}{2}\,\omega''(\abs{x-z})(h_x-h_z)_V^2\,\varepsilon^2+\frac{C}{2}\,\frac{\omega'(\abs{x-z})}{\abs{x-z}}(h_x-h_z)_{V^\bot}^2\,\varepsilon^2 \\
				~& + (4M+1)\,\abs{x-z}^{\gamma-2}\varepsilon^2.
\end{split}
\end{multline}
For the $\beta(x)$-term,
\begin{equation*}
				f(x+\varepsilon\,\zeta,z+\varepsilon\,\zeta) - f(x,z) \leq  4M\prodin{x+z}{\zeta}\,\varepsilon + (4M+1)\,\abs{x-z}^{\gamma-2}\varepsilon^2.
\end{equation*}
Integrating over $\B$ the first order term vanishes, then,
\begin{equation}\label{Classic-[II]}
				\dashint_\B f(x+\varepsilon\,\zeta,z+\varepsilon\,\zeta)\ d\zeta -f(x,z) \leq (4M+1)\,\abs{x-z}^{\gamma-2}\varepsilon^2.
\end{equation}
Finally, for the last term in \eqref{Classic-F-def},
\begin{multline*}
				f(x+\varepsilon h_x,z+\varepsilon\,\zeta) - f(x,z) \\
\begin{split}
				\leq ~&  C\,\omega'(\abs{x-z})\,(h_x-\zeta)_V\,\varepsilon + 2M\prodin{x+z}{h_x+\zeta}\,\varepsilon \\
				~& + \frac{C}{2}\,\omega''(\abs{x-z})(h_x-\zeta)_V^2\,\varepsilon^2 +\frac{C}{2}\,\frac{\omega'(\abs{x-z})}{\abs{x-z}}(h_x-\zeta)_{V^\bot}^2\,\varepsilon^2 \\
				~& + (4M+1)\,\abs{x-z}^{\gamma-2}\varepsilon^2.
\end{split}
\end{multline*}
Due to symmetry, the first order terms containing $\zeta$ cancel out after integration over $\B$, while for the second order terms, we use the rough estimate $\omega''\leq 0$. For the remaining term, we develop $(h_x-\zeta)_{V^\bot}^2$ using notation \eqref{hV},
\begin{equation*}
\begin{split}
				(h_x-\zeta)_{V^\bot}^2
				~& =\abs{h_x-\zeta}^2-(h_x-\zeta)_V^2 \\
				~& =\abs{h_x}^2-(h_x)_V^2+\abs{\zeta}^2-\zeta_V^2-2\brackets{\prodin{h_x}{\zeta}-(h_x)_V\zeta_V} \\
				~& =(h_x)_{V^\bot}^2+\zeta_{V^\bot}^2-2\prodin{h_x-(h_x)_V\v}{\zeta}.
\end{split}
\end{equation*}
Note that, again by symmetry, the last term vanishes after integration and, since $\zeta_{V^\bot}^2\leq 1$ for any $\abs{\zeta}\leq 1$, we get
\begin{equation*}
				\dashint_{\B}(h_x-\zeta)_{V^\bot}^2\ d\zeta \leq (h_x)_{V^\bot}^2+1.
\end{equation*}
Therefore,
\begin{multline}\label{Classic-[III]}
				\dashint_{\B}f(x+\varepsilon\,h_x,z+\varepsilon\,\zeta)d \Ln(\zeta) - f(x,z) \\
\begin{split}
				\leq ~&  C\,\omega'(\abs{x-z})(h_x)_V\,\varepsilon + 2M\prodin{x+z}{h_x}\,\varepsilon
				+\frac{C}{2}\,\frac{\omega'(\abs{x-z})}{\abs{x-z}}\brackets{1+(h_x)_{V^\bot}^2}\,\varepsilon^2 \\
				~& + (4M+1)\,\abs{x-z}^{\gamma-2}\varepsilon^2.
\end{split}
\end{multline}

Then, replacing \eqref{Classic-[I]}, \eqref{Classic-[II]} and \eqref{Classic-[III]} in \eqref{Classic-F-def} we get \eqref{Classic-F-expansion}.
\end{proof}

Note that \Cref{LEMMA-F-Classic} is the analogous version of \Cref{LEMMA-F} in the case $1<p(x)\leq\infty$. Then, the next step is to show the key \Cref{PROPOSITION} for the function $F$ defined in \eqref{Classic-F-def}. In fact, since $\beta(x)\abs{h_x+h_z}^2\geq 0$, the expansion for $F$, \eqref{Classic-F-expansion}, is smaller than
\begin{equation*}
\begin{split}
				F(x,z,h_x,h_z)-f(x,z)
				\leq ~& C\,\omega'(\abs{x-z})\brackets{\alpha(z)(h_x-h_z)_V+(\alpha(x)-\alpha(z))(h_x)_V}\,\varepsilon \\
				~& + 2M\prodin{x+z}{\alpha(z)(h_x+h_z)+(\alpha(x)-\alpha(z))h_x}\,\varepsilon \\
				~& + \frac{C}{2}\,\omega''(\abs{x-z})\,\alpha(z)(h_x-h_z)_V^2\,\varepsilon^2 \\
				~& +\frac{C}{2}\,\frac{\omega'(\abs{x-z})}{\abs{x-z}}\Big\{\alpha(z)(h_x-h_z)_{V^\bot}^2+\beta(x)\abs{h_x+h_z}^2 \\
				~& \hspace{80pt}+(\alpha(x)-\alpha(z))\brackets{1+(h_x)_{V^\bot}^2}\Big\}\varepsilon^2 \\
				~& + (4M+1)\,\abs{x-z}^{\gamma-2}\varepsilon^2,
\end{split}
\end{equation*}
which contains exactly the same terms as in \eqref{F-expansion}, its analogous in \Cref{Section2}. Thus, proceeding exactly as in \Cref{Section2}, we prove the key lemma in the case $\abs{x-z}>\frac{N}{10}\varepsilon$. Finally, repeating the same argument from \Cref{remarks}, we show the key lemma in the case $\abs{x-z}\leq\frac{N}{10}\varepsilon$, and thus we conclude the proof of \Cref{MAIN-THM-Classic}.

Observe that the above proof can be modified to have stability when $p(x)$ is close to $2$. To this end we should use a mirror point coupling for the noise term, as it is done in \cite{luirop} in the case of the Hölder regularity. However, for consistency with the previous sections, we have made this expository choice here.


\appendix
\section{Orthogonal transformations}\label{Appendix A}

\begin{lemma}\label{matrices}
Let $\abs{\nu_x}=\abs{\nu_z}=1$. There exist $\P_{\nu_x},\P_{\nu_z}\in O(n)$ such that $\P_{\nu_x}\e_1=\nu_x$, $\P_{\nu_z}\e_1=\nu_z$ and 
\begin{equation*}
				\abs{\P_{\nu_x}\zeta-\P_{\nu_z}\zeta}\leq\abs{\nu_x+\nu_z}
\end{equation*}
for every $\zeta\in\B^{\e_1}$.
\end{lemma}

\begin{proof}
In order to show this result, we construct explicit orthogonal matrices satisfying the required conditions. For each fixed 
$\abs{\nu_x}=\abs{\nu_z}=1$, we choose $\P_{\nu_x}$ and $\P_{\nu_z}$ in $O(n)$ as follows. First, we denote by $\braces{\nu_x}^\bot$, $\braces{\nu_z}^\bot$ and $\braces{\nu_x,\nu_z}^\bot$ the vector spaces
\begin{equation*}
\begin{split}
                &\braces{\nu_x}^\bot:\,=\set{\xi\in\R^n}{\prodin{\nu_x}{\xi}=0},\\
                &\braces{\nu_z}^\bot:\,=\set{\xi\in\R^n}{\prodin{\nu_z}{\xi}=0},\\
                &\braces{\nu_x,\nu_z}^\bot:\,=\braces{\nu_x}^\bot\cap\braces{\nu_z}^\bot.
\end{split}
\end{equation*}
Then $\dim\braces{\nu_x}^\bot=\dim\braces{\nu_z}^\bot=n-1$. If $\nu_x=\pm\nu_z$, then $\braces{\nu_x}^\bot=\braces{\nu_z}^\bot$, otherwise $\dim\braces{\nu_x,\nu_z}^\bot=n-2$. In both cases, we can find a $(n-2)$-dimensional vector space contained in $\braces{\nu_x,\nu_z}^\bot$. Then, let $\braces{r_3,r_4,\ldots,r_n}$ be a collection of $n-2$ unitary column vectors in $\R^n$ that form an orthonormal basis for such subspace. Let $\RR\in\R^{n\times(n-2)}$ be the matrix containing all the elements of the basis as column vectors, i.e.,
\begin{equation*}
                \RR:\,=\brackets{\ r_3 \  r_4 \  \cdots \ r_n\ }.
\end{equation*}
Note that, therefore, the vector space
\begin{equation*}
					\braces{\RR}^\bot:\,=\set{\xi\in\R^n}{\RR^\top\xi=0}
\end{equation*}
defines a ($2$-dimensional) plane containing the unitary vectors $\nu_x$ and $\nu_z$. In addition, for $\nu_x\in\braces{\RR}^\bot$, there exist a unique unitary vector $\varrho_x\in\braces{\RR}^\bot\cap\braces{\nu_x}^\bot$ such that
\begin{equation*}
				\P_{\nu_x}=\brackets{\ \nu_x \ \varrho_x \ \RR \ }\in O(n) \quad\quad \mbox{ and } \quad \quad \det\P_{\nu_x}=1.
\end{equation*}
Analogously, let $\varrho_z\in\braces{\RR}^\bot\cap\braces{\nu_z}^\bot$ the unique unitary vector such that
\begin{equation*}
				\P_{\nu_z}=\brackets{\ \nu_z \ \varrho_z \ \RR \ }\in O(n) \quad\quad \mbox{ and } \quad \quad \det\P_{\nu_z}=-1.
\end{equation*}
Then,
\begin{equation*}
				\P_{\nu_x}-\P_{\nu_z}=\brackets{\ \nu_x-\nu_z  \ \ \varrho_x-\varrho_z \ \ \mathbf{0}\ },
\end{equation*}
and, for any $\zeta\in\B^{\e_1}$, $\zeta_1=0$ and 
\begin{equation*}
\begin{split}
				\abs{\P_{\nu_x}\zeta-\P_{\nu_z}\zeta}
				~& = \abs{\zeta_2(\varrho_x-\varrho_z)} \leq \abs{\varrho_x-\varrho_z}.
\end{split}
\end{equation*}

Finally, we show that, for this particular choice of the vectors $\varrho_x$ and $\varrho_z$, it holds
\begin{equation*}
				\abs{\varrho_x-\varrho_z}=\abs{\nu_x+\nu_z}.
\end{equation*}

By the properties of the $n$-dimensional orthogonal group, the matrix $\P_{\nu_x}^\top\P_{\nu_z}$ is also in $O(n)$ with determinant $\det(\P_{\nu_x}^\top\P_{\nu_z})=-1$, and it takes the form
\begin{equation*}
                \P_{\nu_x}^\top\P_{\nu_z}
                =\brackets{\begin{array}{cc}
                \Q & \mathbf{0} \\
                \mathbf{0} & \I_{n-2}
                \end{array}},
\end{equation*}
where
\begin{equation*}
				\Q=\brackets{\begin{array}{cc}\prodin{\nu_x}{\nu_z} & \prodin{\nu_x}{\varrho_z} \\
                \prodin{\varrho_x}{\nu_z} & \prodin{\varrho_x}{\varrho_z}
                \end{array}}\in O(2)
\end{equation*}
has determinant $\det\Q=-1$, that is, $\Q$ is a reflection matrix in $\R^2$ and, thus, there exists $\sigma\in[0,2\pi)$ such that 
\begin{equation*}
                \Q=\brackets{\begin{array}{cc}\sin\sigma & \cos\sigma \\
                \cos\sigma & -\sin\sigma
                \end{array}}.
\end{equation*}
Then, in particular, $\prodin{\varrho_x}{\varrho_z}=-\prodin{\nu_x}{\nu_z}$ and
\begin{equation*}
				\abs{\varrho_x-\varrho_z}^2=2-2\prodin{\varrho_x}{\varrho_z}=2+2\prodin{\nu_x}{\nu_z}=\abs{\nu_x+\nu_z}^2. \qedhere
\end{equation*}
\end{proof}


\end{document}